\documentclass[review,onefignum,onetabnum]{siamart190516}

\newcommand{\Ctwo}{\ensuremath{\mathcal{C}^2 }}

\usepackage{lipsum}
\usepackage{amsfonts}
\usepackage{graphicx}
\usepackage{epstopdf}
\usepackage[noend]{algorithmic}
\ifpdf
  \DeclareGraphicsExtensions{.eps,.pdf,.png,.jpg}
\else
  \DeclareGraphicsExtensions{.eps}
\fi

\newsiamremark{remark}{Remark}
\newsiamremark{hypothesis}{Hypothesis}
\crefname{hypothesis}{Hypothesis}{Hypotheses}
\newsiamthm{claim}{Claim}
\newtheorem{assumption}[theorem]{Assumption}
\newtheorem{example}{Example}

\headers{Stochastic optimization for problems with upper-$\Ctwo$ objectives}{J. Wang, I. Aravena and C. G. Petra}

\title{A sequential quadratic programming method for nonsmooth stochastic optimization with upper-$\Ctwo$ objective
        \thanks{Submitted to the editors. LLNL release number LLNL-JRNL-846514.
\funding{Prepared by Lawrence Livermore National Laboratory (LLNL) under Contract DE-AC52-07NA27344. }}}

\author{Jingyi Wang\thanks{Center for Applied Scientific Computing, Lawrence Livermore National Laboratory,
Livermore, CA 
  (\email{wang125@llnl.gov, petra1@llnl.gov}).}
\and Ignacio Aravena \thanks{Computational Engineering Division, Lawrence Livermore National Laboratory,
Livermore, CA 
  (\email{aravenasolis1@llnl.gov}).}
\and Cosmin G. Petra\footnotemark[2]
}

\usepackage{amsopn}

\ifpdf
\hypersetup{
  pdftitle={Stochastic optimization algorithm for nonsmooth nonconvex problems with upper-$\Ctwo$ objective},
  pdfauthor={J. Wang and I. Aravena and C. G. Petra}
}
\fi
\usepackage{mathtools}

\begin{document}
\nolinenumbers
\maketitle

\begin{abstract}
    We propose a sequential quadratic programming (SQP) method that can incorporate adaptive sampling for stochastic nonsmooth nonconvex optimization problems with upper-$\Ctwo$ objectives. Upper-$\Ctwo$ functions can be viewed as difference-of-convex (DC) functions with smooth convex parts. They are common among certain classes of solutions to parametric optimization problems, \textit{e.g.}, recourse of stochastic programming and closest-point projection onto closed sets. Our proposed algorithm is a stochastic SQP with line search and bounded algorithmic parameters and is shown to achieve subsequential convergence in expectation for nonsmooth problems with upper-$\Ctwo$ objectives. We discuss various sampling strategies, including an adaptive sampling one, that can potentially improve algorithm efficiency. 
The capabilities of our algorithm are demonstrated by solving a joint production, pricing and shipment problem, as well as a realistic optimal power flow problem as used in current power grid industry practice.
\end{abstract}



\newcommand{\Rbb}{\ensuremath{\mathbb{R} }}
\newcommand{\Nbb}{\ensuremath{\mathbb{N} }}
\newcommand{\Ebb}{\ensuremath{\mathbb{E} }}
\newcommand{\Pbb}{\ensuremath{\mathbb{P} }}
\section{Introduction}\label{se:intro}

In this paper, we consider the class of nonsmooth nonconvex  optimization problems in the form of 
\begin{equation} \label{eqn:opt0}
 \centering
  \begin{aligned}
	  &\underset{\substack{x}\in C}{\text{minimize}} 
	  & & f(x)+ r(x), \\
          &\text{subject to}
	  & & c(x) = 0,
  \end{aligned}
\end{equation}
where the objective function $f:\Rbb^n\to\Rbb$ and equality constraint function $c:\Rbb^n\to\Rbb^m$ are continuously differentiable with Lipschitz gradients and $C \subset \Rbb^n$ is a convex compact set.
The stochastic part of the objective is given as $r(x)=\mathbb{E}_{\xi}\left[R(x,\xi)\right]$, where 
the expectation $\mathbb{E}_{\xi}$ is the expected value of the function $R(x,\xi)$ with respect to $\xi$.
The random vector $\xi$ is defined over a probability space $(\Omega,\mathcal{F},\mathbb{P})$, where $\Omega$ is the sample space, $\mathcal{F}$ the $\sigma$-algebra generated by the subspace of $\Omega$ and $\mathbb{P}$ a probability measure on $\mathcal{F}$.  
In this work, we focus on $r:\Rbb^n \to\Rbb$ 
that is Lipschitz continuous and  
upper-$\Ctwo$ (see monograph~\cite{rockafellar1998} and Section~\ref{sec:prob}), but potentially nonsmooth. 

Upper-$\Ctwo$ functions are a subset of DC functions, and is equivalent to weakly concave functions on a convex compact set in $\Rbb^n$~\cite{cui2021}. 
Any finite, concave function is upper-$\Ctwo$~\cite[Theorem~10.33]{rockafellar1998}, as are all continuously differentiable functions~\cite[Proposition~13.34]{rockafellar1998}.
Moreover, a squared distance function to a closed set, which is  
the value function of a minimization problem, is upper-$\Ctwo$~\cite[Example~10.57]{rockafellar1998}.
The property exists in many important two-stage stochastic programming problems with recourse of the form of~\eqref{eqn:opt0}~\cite{Shapiro_book,Birge97Book,KallWallace}, 
whose first-stage objective includes an upper-$\Ctwo$ value function $r$ of the second-stage problem.
The application that motivated us to look into this property is 
the security-constrained alternating current optimal power flow (SCACOPF) problem~\cite{ChiangPetraZavala_14_PIPSNLP,Qiu2005,petra_14_augIncomplete,petra_14_realtime,petra_21_gollnlp}, which can be stated as a stochastic two-stage problem. 
In this case, the nonsmooth part of the first-stage objective $r$ becomes upper-$\Ctwo$ through regularization of the second-stage problems, if it is not already so~\cite{wang2022}. 
We add that upper-$\Ctwo$ does not guarantee differentiability, or (lower) regularity (see Section~\ref{sec:prob}). 

The most successful nonsmooth optimization methods include subgradient methods~\cite{shor1985}, bundle methods~\cite{makela1992,Kiwiel1996}, and DC algorithms~\cite{an2018}. 
In particular, many bundle methods have been shown to enjoy asymptotic global convergence with lower-$\mathcal{C}^1$ or upper-$\mathcal{C}^1$ objectives~\cite{lemarechal1996,lemarechal1978,schramm1992,noll2009,hare2010,noll2013}. 
For constrained problems, bundle methods typically incorporate constraints into the objective through penalty or filter methods so that the bundle approximation is constructed for the new objective~\cite{hare2015,dao2016}. The extension to stochastic optimization is challenging as the piecewise-linear cutting planes are no longer exact. However, if the error from the finite approximation to the stochastic objective is bounded, bundle methods have been proven to converge to proximity of certain optimality points with inexact information~\cite{oliveira2011,hare2015,noll2013}. 

Recently, a significant body of literature has been dedicated to research of (nonsmooth) stochastic optimization methods, where prominent algorithms include stochastic gradient descent~\cite{bottou2012stochastic} and stochastic subgradient descent~\cite{davis2019,ram2009incremental,nedic2014stochastic}. 
For constrained smooth stochastic problems, sequential quadratic programming (SQP) has been applied to stochastic objective and deterministic constraints, including both equality and inequality constraints~\cite{berahas2021sqpsto}. 
Nonsmooth stochastic DC algorithms (DCA) have garnered great interest as well, typically through combining DCA with sample average approximation (SAA)~\cite{liu2020nonsmooth,le2022stochastic,liu2020}. 
The smoothing of the objective is explored in~\cite{kouri2022primaldual}, where a primal-dual regularization of the objective combined with SAA provides a solution to problems with certain nonsmooth risk measures. 

In addition to SAA, adaptive sampling, which aims to improve efficiency of algorithms by adaptive sample size, is an active research topic for stochastic gradient descent, augmented Lagrangian, and SQP methods~\cite{byrd2012,berahas2022adaptive}. 
The sample size is commonly controlled via variance in the stochastic approximation of quantities such as the gradient. In the unconstrained setting, a well-established norm test~\cite{byrd2012,carter1991} can be used.
Its theory and application to smooth objectives on convex feasible regions has been shown in~\cite{berahas2022adaptive}. 

To the best of our knowledge, algorithm development for stochastic nonsmooth constrained problems, particularly with adaptive sampling, has been limited. 
In~\cite{wang2022}, we described a nonsmooth SQP method for upper-$\Ctwo$ objectives. The method uses line search on the constraints and trust-region update rule to achieve convergence with practical success in solving SCACOPF~\cite{wang2021}. In this paper, we extend the effort to stochastic optimization problems. In the stochastic setting, the trust-region update rule and serious step criterion are no longer available without exact function values. We point out that since both DCAs and our stochastic SQP linearizes the nonconvex part of the objective, the reader can also view the proposed SQP algorithm as a version of DCA designed for constrained problems with upper-$\Ctwo$ objectives. 

This paper is organized as follows. In Section~\ref{sec:prob}, we describe the mathematical notations and background necessary for the stochastic optimization algorithm.
In Section~\ref{sec:alg-convex}, we propose an algorithm for a simplified version of problem~\eqref{eqn:opt0} 
where the equality constraints $c(x)=0$ are neglected, making line search unnecessary. We provide subsequential convergence analysis 
for the algorithm with probability 1, under assumptions typically made for stochastic and nonsmooth optimization.
In Section~\ref{sec:alg-line-search}, we expand the algorithm to the full problem~\eqref{eqn:opt0} based on 
a line search algorithm the authors previously developed~\cite{wang2022}. We show subsequential convergence with probability 1 of the algorithm 
to Karush–Kuhn–Tucker (KKT) points (points that satisfy KKT conditions) under proper assumptions. 
In section~\ref{sec:adp-sample}, we present adaptive sampling criteria in theory and implementation that could lead to more efficient results in practice. 
Numerical experiments are shown in Section~\ref{sec:exp} to illustrate the practical capabilities of the proposed algorithm.
Finally, Section~\ref{sec:con} concludes the paper and suggests directions for future research.

\newcommand{\norm}[1]{\left\lVert {#1} \right\rVert}
\section{Background and notations}\label{sec:prob}

In this section, we discuss the mathematical background and notations necessary for the analysis in this paper. 
First, the lower regular subdifferential of a function $r:\Rbb^n\to\Rbb$ at point $\bar{x}$, 
denoted as $\hat{\partial} r(\bar{x})$, is defined by  
\begin{equation}\label{eqn:subgradient-def}
\centering
 \begin{aligned}
	\hat{\partial} r (\bar{x}) = \left\{ g\in\Rbb^n |\liminf_{ \substack{x\to \bar{x}\\x\neq \bar{x} }}\frac{r(x)-r(\bar{x})-\langle g,x-\bar{x}\rangle  }{\norm{x-\bar{x}}}\geq 0\right\},
 \end{aligned}
\end{equation}
where $\norm{\cdot}$ is the 2-norm and $\langle \cdot \rangle$ is the inner product in $\Rbb^n$~\cite{rockafellar1998}.
For a sequence $\{x^{\nu}\}$, the notation $x^{\nu}\to x$ means $\norm{x^{\nu}-x}\to 0$ as $\nu\to\infty$.
If there exists a sequence $\{x^{\nu}\}$ such that $x^{\nu} \xrightarrow[r]{} \bar{x}$ and $g^{\nu}\in \hat{\partial}r(x^{\nu})$ with 
$\{g^{\nu}\}$ converging $g^{\nu}\to \bar{g}\in\Rbb^n$, then 
$\bar{g}$ is a lower general subgradient of $r(\bar{x})$. The f-attentive convergence, denoted as  
$ x^{\nu} \xrightarrow[r]{} \bar{x}$, is  
\begin{equation}\label{def:f-attentive}
\centering
 \begin{aligned}
	 x^{\nu} \xrightarrow[r]{} \bar{x} \quad \ratio\Leftrightarrow \quad x^{\nu}\to\bar{x} \quad \text{with} \quad r(x^{\nu})\to r(\bar{x}),
 \end{aligned}
\end{equation}
and is trivially satisfied for $x^{\nu}\to \bar{x}$ when $r$ is Lipschitz (see 8(2),~\cite{rockafellar1998} for details). 
The lower general subdifferential at $\bar{x}$ is denoted as $\partial r(\bar{x})$.

A Lipschitz function $r$ is lower regular (or subdifferentially regular) 
if and only if $\partial r(\bar{x})=\hat{\partial} r(\bar{x})$~\cite[Corollary~8.11]{rockafellar1998}. 
Due to its importance, lower general subgradient is often simply called general subgradient, 
while a lower regular function is called regular. 
On the other hand, 
upper regular subdifferential~\cite{rockafellar1998,mordukhovich2004upp} is defined as
\begin{equation}\label{def:upp-subgradient}
\centering
 \begin{aligned}
	 \hat{\partial}^+ r(\bar{x}) =& - \hat{\partial} (-r)(\bar{x})
	 =\left\{ g\in\Rbb^n |\limsup_{ \substack{x\to \bar{x}\\x\neq \bar{x} }}\frac{r(x)-r(\bar{x})-\langle g,x-\bar{x}\rangle  }{\norm{x-\bar{x}}}\leq 0\right\}.
 \end{aligned}
\end{equation}
The upper general subdifferential is $\partial^+ r(\bar{x}) = -\partial (-r)(\bar{x})$.
A function $r$ is called upper regular if $-r$ is lower regular. 

In nonsmooth optimization literature, both Clarke subdifferential~\cite{clarke1983}, denoted as $\bar{\partial} r{(\bar{x})}$ of $r$ at $\bar{x}$, and $\partial r(\bar{x})$ have been widely adopted.
If $r$ is lower regular, $\partial r(\bar{x}) =\bar{\partial} r{(\bar{x})}$~\cite[Theorem~8.6,~8.49,~9.61]{rockafellar1998}.
The same holds for an upper regular $r$ as well (for proof see~\cite{wang2022}) and lets one use its Clarke subgradient and upper general subgradient interchangeably. Thus, in the remainder of the paper, we simply refer to Clarke subgradient as subgradient.
Clarke subdifferential is  outer/upper-semicontinuous, which is necessary in establishing subsequential convergence~\cite[Proposition~6.6]{rockafellar1998}. 
In addition, for a Lipschitz $r$, $\bar{\partial} r(\bar{x})$ is locally bounded~\cite[Theorem~9.13]{rockafellar1998}. 
Next, we discuss lower-$\mathcal{C}^k$, introduced in~\cite{Spingarn1981SubmonotoneSO,rockafellar1998}
with equivalent definitions in~\cite{daniilidis2004}. 
A function $r:O\to \Rbb$, where $O\subset\Rbb^n$ is open, is said to be lower-$\mathcal{C}^k$ on $O$, if on some neighborhood $V$ of each $x\in O$ there is a representation
\begin{equation}\label{eqn:lowc1-def-0}
\centering
 \begin{aligned}
	 r(x) =  \underset{\substack{t \in T}}{\text{max}} \ r_t(x),
 \end{aligned}
\end{equation}
where $r_t:\Rbb^n\to\Rbb$ is of class $\mathcal{C}^k$ on $V$ and the index set $T$ is a compact space such that $r_t$ 
and all of its partial derivatives through order $k$ are continuous on $(t,x)\in T\times V$. 
A function is upper-$\mathcal{C}^k$ if we replace the $max$ with $min$ in~\eqref{eqn:lowc1-def-0}.
Let $T\subset\Rbb^p$ be compact, the function $r$ is upper-$\Ctwo$ if it can be expressed as
\begin{equation}\label{eqn:uppc2-def-1}
\centering
 \begin{aligned}
	 r(x) =  \underset{\substack{t \in T}}{\text{min}} \ p(t,x)
 \end{aligned}
\end{equation}
for all $x\in O$, such that $p(\cdot,\cdot):\Rbb^p\times\Rbb^n\to\Rbb$ and its first- and second-order partial derivatives in $x$ 
are continuous on $(t,x)$. Clearly, upper-/lower-$\mathcal{C}^k$ imply upper-/lower-regularity.

Equivalently, a finite-valued function $r$ is lower-$\Ctwo$ on $O\subset\Rbb^n$ if and only if there exists $\rho>0$ such that 
$r(\cdot)+\frac{1}{2}\rho\norm{\cdot}^2$ is convex. 
Notice that this definition is already given on $O$ with a uniform $\rho$. 
If $r$ is upper-$\Ctwo$, \textit{e.g.}, $r(x)=-|x|,x\in\Rbb$, by definition, $-r$ is lower-$\Ctwo$. 
Through simple arithmetic and convexity with subgradients, we have 
\begin{equation}\label{eqn:uppc2-def}
\centering
 \begin{aligned}
	 r(x) - r(\bar{x}) - \langle g,x -\bar{x} \rangle \leq \frac{\rho}{2}\norm{x-\bar{x}}^2, 
 \end{aligned}
\end{equation}
where $g \in \bar{\partial} r(\bar{x})$.
We refer to~\eqref{eqn:uppc2-def} as the upper-$\Ctwo$ inequality. 
Without loss of generality,~\eqref{eqn:opt0} can be simplified to
\begin{equation} \label{eqn:opt-prob-sto}
 \centering
  \begin{aligned}
   &\underset{\substack{x \in C}}{\text{minimize}} 
	  & & r(x)=\Ebb\left[R(x,\xi) \right],\\
	  &\text{subject to} &&c(x) = 0.
  \end{aligned}
\end{equation}
There are multiple optimality conditions used for nonsmooth constrained problems, \textit{e.g.}, 
stationary point, KKT point, critical point. 
In this paper, we assume linear independent constraint qualification (LICQ, Section~\ref{sec:alg-line-search}) 
at local minima and accumulation points of the proposed algorithms for~\eqref{eqn:opt-prob-sto}, due to presence of both equality 
and inequality constraints. This enables us 
to establish the necessary optimality condition in terms of subgradient~\cite[Section~6.4]{clarke1983}.  
For problem~\eqref{eqn:opt-prob-sto}, at a local minimum $\bar{x}$, there exists $\bar{\lambda} \in \Rbb^m $ such that 
\begin{equation} \label{eqn:opt-sto-KKT}
   \centering
  \begin{aligned}
	  0 \in  \bar{\partial} r(\bar{x}) + \nabla c (\bar{x}) \bar{\lambda} + \bar{\partial} i_C(\bar{x})&, \\
	  c_j(\bar{x}) = 0,j=1,\dots,m, \ \bar{x}\in C&,
  \end{aligned}
\end{equation}
where the matrix $\nabla c(\bar{x})$ is of dimension $n\times m$, and $\bar{\partial} i_C(\cdot)$ is
the (Clarke) subdifferential of the indicator function $i_C(\cdot)$ defined as
\begin{equation} \label{eqn:def-ind}
 \centering
  \begin{aligned}
          i_C(x)  =   \begin{cases}
                  0,  &\text{if} \ x \in C, \\ 
                  +\infty, &\text{otherwise.}
    \end{cases}
  \end{aligned}
\end{equation}
Note that because $C$ is convex, so is $i_C(\cdot)$. 
We call a point that satisfies~\eqref{eqn:opt-sto-KKT} a KKT point of~\eqref{eqn:opt-prob-sto}.
Since~\eqref{eqn:opt-sto-KKT} is defined with Clarke subdifferential, a point that satisfies it can be  referred to as a Clarke stationary point~\cite{cui2021}.

\section{\normalsize Stochastic optimization with convex set constraint}\label{sec:alg-convex}
In this section, we focus on a simplified problem
\begin{equation} \label{eqn:opt-prob-sto-convex}
 \centering
  \begin{aligned}
   &\underset{\substack{x \in C}}{\text{minimize}} 
	  & & r(x)=\Ebb\left[R(x,\xi) \right],
  \end{aligned}
\end{equation}
which permits simpler algorithms and assumptions necessary for convergence.
Motivated by the discussion in section~\ref{sec:prob}, two assumptions are formalized below.
\begin{assumption}\label{assp:upperC2}
	The Lipschitz continuous objective $r$ in~\eqref{eqn:opt-prob-sto-convex} is upper-$\Ctwo$. 
\end{assumption}
\noindent In particular, inequality~\eqref{eqn:uppc2-def} is satisfied. 
\begin{assumption}\label{assp:boundedHc}
   The feasible region $C$ is convex and compact.
\end{assumption}
Here, we opt for a compact $C$ instead of a closed $C$ with bounded iterates $\{x_k\}$ for ease of presentation.
A bounded $C$ would ensure bounded $\{x_k\}$ and given Assumption~\ref{assp:upperC2} guarantee a bounded $r(\cdot)$ as well. 
Therefore, there exists $r_{\min}$ such that $r_{\min}\leq r(x), \forall x\in C$.
It is assumed that $x\in C$ can be enforced in the optimization subproblems, \textit{e.g.}, through projection onto $C$.

\subsection{Algorithm description}\label{sec:alg-convex-description}
At iteration $k$ and its iterate $x_k$, the local approximation model to the objective with a true subgradient $g_k\in \bar{\partial} r(x_k)$ is
\begin{equation} \label{eqn:opt-true-appx}
 \centering
  \begin{aligned}
	  \phi_k(x) = r(x_k) + g_k^T(x-x_k) + \frac{1}{2}\alpha_k \norm{x-x_k}^2,
  \end{aligned}
\end{equation}
where $\alpha_k>0$ is a scalar quadratic coefficient. 
In the stochastic setting, however, we use a stochastic subgradient estimate $\bar{g}_k$  of $g_k$ and the local model $\bar{\phi}_k(\cdot)$ is
\begin{equation} \label{eqn:opt-sto-appx}
 \centering
  \begin{aligned}
	  \bar{\phi}_k(x) = \bar{r}_k + \bar{g}_k^T(x-x_k) + \frac{1}{2}\alpha_k \norm{x-x_k}^2.
  \end{aligned}
\end{equation}
The stochastic estimate of function value $\bar{r}_k$ does not affect the optimization subproblem and its solution. Hence, we hereby use it interchangeably with $r(x_k)$. 
Denoting $d=x-x_k$, $\bar{\phi}_k(x)$ and $\phi_k(x)$ can be rewritten as $\bar{\Phi}_k(d)$ and $\Phi_k(d)$, respectively, 
\begin{equation} \label{eqn:opt-sto-appx-d}
 \centering
  \begin{aligned}
	  \bar{\Phi}_k(d) = \bar{r}_k + \bar{g}_k^T d +\frac{1}{2}\alpha_k\left\lVert d\right\rVert ^2, \ 
	  \Phi_k(d) = r(x_k) + g_k^Td + \frac{1}{2}\alpha_k \norm{d}^2.
  \end{aligned}
\end{equation}
%
The subproblem to be solved at iteration $k$ and $x_k$ is thus  
\begin{equation} \label{eqn:opt-sto-convex}
 \centering
  \begin{aligned}
   &\underset{\substack{x_k+d\in C}}{\text{minimize}} 
	  & & \bar{\Phi}_k(d).
  \end{aligned}
\end{equation}
Meanwhile, the corresponding subproblem with true subgradient is
\begin{equation} \label{eqn:opt-true-convex}
 \centering
  \begin{aligned}
   &\underset{\substack{x_k+d\in C}}{\text{minimize}} 
	  & & \Phi_k(d).
  \end{aligned}
\end{equation}
Let $\bar{d}_k$ be the solution to~\eqref{eqn:opt-sto-convex} and $d_k$ the solution to~\eqref{eqn:opt-true-convex}.
The predicted change on the objective is defined as 
\begin{equation} \label{def:pd}
 \centering
  \begin{aligned}
	  \bar{\delta}_k =& \bar{\Phi}_k(0) - \bar{\Phi}_k(\bar{d}_k) 
	   = -\bar{g}_k^T \bar{d}_k -\frac{1}{2}\alpha_k\norm{\bar{d}_k}^2.
  \end{aligned}
\end{equation}
Similarly, $\delta_k =  -g_k^T d_k -\frac{1}{2}\alpha_k\norm{d_k}^2$.
The choice of $\alpha_k$ is independent of the stochastic subgradient.
We rely on knowledge of $r(\cdot)$ such as its Lipschitz constant 
to select $\alpha_k$ such that $\alpha_k$ is large enough to ensure convergence in expectation, in line with stochastic optimization literature~\cite{berahas2021sqpsto}. 
The first-order optimality conditions of subproblem~\eqref{eqn:opt-sto-convex} are
\begin{equation} \label{eqn:sto-convex-KKT}
  \centering
   \begin{aligned}
	   &\bar{g}_k + \alpha_k \bar{d}_k + \bar{v}_k =0,\\
	   &x_k+\bar{d}_k \in C,
	   \bar{v}_k \in \bar{\partial} i_C (x_k+\bar{d}_k).
   \end{aligned}
 \end{equation}
The first-order optimality conditions of subproblem~\eqref{eqn:opt-true-convex} are
\begin{equation} \label{eqn:true-convex-KKT}
  \centering
   \begin{aligned}
	   &g_k + \alpha_k d_k + v_k  =0,\\
	   &x_k+d_k \in C,
	   v_k \in \bar{\partial} i_C (x_k+ d_k).
   \end{aligned}
 \end{equation}

The stochastic quantities are estimated at each iteration with independent and identically distributed (i.i.d.) random variables $\xi$.
 We use $S_k$ to denote the set of samples of $\xi$  at iteration $k$ with realizations $\xi_i,\xi_i\in S_k$. 
The cardinality or sample size of $S_k$ is defined as $N_k = |S_k|$.

The simplified stochastic SQP is presented in Algorithm~\ref{alg:sto-convex}.
While we require $\{\alpha_k\}$ to be monotonically non-decreasing, in step $7$, for a simpler analysis, it is possible to relax this requirement as long as $\rho\leq\alpha_k\leq \eta_{\alpha}\rho$ is maintained.
Additionally, the quadratic term $\alpha_k\norm{d_k}^2$ can be replaced by $d_k^T B_k d_k$, where $B_k$ is a positive definite matrix. We leave the exact update rule for $\alpha_k$ to be specified for each application.
\begin{algorithm}
   \caption{Simplified stochastic nonsmooth SQP}\label{alg:sto-convex}
	\begin{algorithmic}[1]
	    \STATE{Choose scalar $\eta_{\alpha}>1$.
		  Initialize $x_0$, $\rho\leq\alpha_0\leq\eta_{\alpha}\rho$ and $k=0$.
		  Choose initial sampling set $S_0$.
		}
	\FOR{$k=0,1,2,...$}
                \STATE{Generate sample sets $\{\xi_i\},\xi_i\in S_k$ i.i.d. from probability distribution of $\xi$.}
		\STATE{Evaluate the function value $\bar{r}_k$ and stochastic subgradient estimate $\bar{g}_k$.}
		\STATE{Form the objective $\bar{\Phi}_k$ in~\eqref{eqn:opt-sto-appx-d} and solve 
		subproblem~\eqref{eqn:opt-sto-convex} to obtain $\bar{d}_{k}$.} 
		\STATE{Take the step $x_{k+1} = x_k + \bar{d}_k$.}
		\STATE{Call the $\alpha_k$ update rules to obtain $\alpha_{k+1}\in [\alpha_k,\eta_{\alpha}\rho]$.}
       \ENDFOR
    \end{algorithmic}
\end{algorithm}
\begin{remark}\label{rmrk:bigalpha}
	Our analysis presumes that $\rho$ is known. We recognize that might not be the case. Therefore, in practical applications, an estimate of $\rho$ is often needed for the choice of $\alpha_k$ such that $\alpha_k\geq\rho$ can be maintained. Some practical techniques for a dynamic estimate of $\rho$ can be found in~\cite{curtis2019exploiting,berahas2021sqpsto,bollapragada2018adaptive}. 
The trust-region update rule for $\alpha_k$ in the deterministic case in~\cite{wang2022} can be extended here if accurate evaluation of function values are available. 
\end{remark}

\subsection{Convergence analysis}\label{sec:alg-convg}
In this section, we assume Assumption~\ref{assp:upperC2},~\ref{assp:boundedHc} are valid.
The choice of $\alpha_k$ in step 7 of Algorithm~\ref{alg:sto-convex} leads to the following Lemma.
\begin{lemma}\label{lem:sto-convex-f-decrease}
	The iterate $x_{k+1}$ from Algorithm~\ref{alg:sto-convex} 
	satisfies $r(x_k)-r(x_{k+1})\geq \Phi_k(0)-\Phi_k(\bar{d}_k)$. 
\end{lemma}
\begin{proof}
	From the upper-$\Ctwo$ inequality~\eqref{eqn:uppc2-def}, we have  
\begin{equation} \label{eqn:sto-convex-f-decrease-p1}
 \centering
  \begin{aligned}
	  r(x_k+d)- r(x_k) - g_{k}^Td \leq \frac{\rho}{2} \norm{d}^2,\\
  \end{aligned}
\end{equation}
	for $g_k\in\bar{\partial}r(x_k)$, $\rho>0$ and	$x\in C$.
Since $\alpha_k$ is monotonically non-decreasing and $\alpha_0 \geq \rho$, we have  
\begin{equation} \label{eqn:sto-convex-f-decrease-p2}
 \centering
  \begin{aligned}
	  \alpha_k \geq \rho, \ \text{for all} \ k\in \Nbb. \\
  \end{aligned}
\end{equation}
	The inequalities~\eqref{eqn:sto-convex-f-decrease-p1} and \eqref{eqn:sto-convex-f-decrease-p2} imply that at $d=\bar{d}_k$
\begin{equation} \label{eqn:sto-convex-f-decrease-obj}
 \centering
  \begin{aligned}
	  r(x_k) - r(x_k+\bar{d}_k) \geq& -g_{k}^T \bar{d}_k -\frac{1}{2}\rho \norm{\bar{d}_k}^2\\
			\geq& -g_{k}^T \bar{d}_k -\frac{1}{2}\alpha_k \norm{\bar{d}_k}^2
				 = \Phi_k(0) -\Phi_k(\bar{d}_k). \\
  \end{aligned}
\end{equation}

\end{proof}

To further the analysis, the estimation error bound in expectation from sampling needs to be established. 
In~\cite{liu2020nonsmooth}, with no specific form of $\bar{\phi}_k(\cdot)$, the variance of the stochastic objectives is assumed to be bounded. 
The variance bounds then comes from sample size $N_k$.
Similarly, in~\cite{berahas2021sqpsto}, an unbiased estimate with bounded variance is made on the gradient with SQP. 
In~\cite{shashaani2018astro}, the Monte Carlo estimate of the objective is assumed to be unbiased and its variance uniformly bounded. 
Given our choice of SQP, without specifying the sampling strategy, we make the assumption that the stochastic estimate of the subgradient is unbiased, and its variance is bounded. 
\begin{assumption}\label{assp:gk}
   For all iterations $k\in \Nbb$, the stochastic subgradient approximation $\bar{g}_k$ is an
    unbiased estimate of $g_k\in\bar{\partial} r(x_k)$, \textit{i.e.}, $\Ebb_k[\bar{g}_k]=g_k$.
    Furthermore, $\bar{g}_k$ satisfies
    \begin{equation} \label{assp:sto-grad}
     \centering
      \begin{aligned}
          \Ebb_k\left[\norm{\bar{g}_k-g_k}^2 \right]\leq M_k ,
      \end{aligned}
    \end{equation}
    where $M_k>0$ for all $k$ is a sequence independent of stochastic solutions. Here, $\Ebb_k$ denotes the expectation over
    $\xi$ conditioned on iterate $x_k$ at the $k$-th iteration, generated by the random vectors in
    $\{S_t\}_{t=0}^k$.
\end{assumption}
\noindent The analysis is performed under Assumption~\ref{assp:gk} throughout the rest of the section. An immediate result of Assumption~\ref{assp:gk} is given in the next lemma.
\begin{lemma}\label{lem:sto-gk-norm}
	For all $k\in \Nbb$ of Algorithm~\ref{alg:sto-convex}, $\Ebb_k\left[ \norm{\bar{g}_k-g_k} \right] \leq \sqrt{M_k} $.
\end{lemma}
\begin{proof}
	From Assumption~\ref{assp:gk} and Jensen's inequality (with square function),  
   \begin{equation} \label{eqn:gk-diff}
     \centering
     \begin{aligned}
	     \Ebb_k[\norm{\bar{g}_k-g_k}] \leq \left( \Ebb_k \left[\norm{\bar{g}_k-g_k}^2 \right]\right)^{\frac{1}{2}} \leq \sqrt{M_k}.
     \end{aligned}
   \end{equation}
\end{proof}
Assumption~\ref{assp:gk} leads to the bounds on the variance of the stochastic search directions.
\begin{lemma}\label{lem:sto-convex-dk-diff}
  For all $k\in \Nbb$ of Algorithm~\ref{alg:sto-convex}, $\Ebb_k[\norm{\bar{d}_k-d_k}] \leq \frac{1}{\alpha_k}\sqrt{M_k}$, and 
	$\Ebb_k[\norm{\bar{d}_k-d_k}^2] \leq \frac{1}{\alpha_k^2} M_k$.
\end{lemma}
\begin{proof}
	From the definition, $d_k$ and $\bar{d}_k$ are the solutions to~\eqref{eqn:opt-true-convex} and~\eqref{eqn:opt-sto-convex}, respectively. 
	Subtracting the optimality conditions~\eqref{eqn:true-convex-KKT} from~\eqref{eqn:sto-convex-KKT} leads to 	
   \begin{equation} \label{eqn:sto-convex-dk-diff-1}
     \centering
     \begin{aligned}
	     \bar{g}_k-g_k + \alpha_k(\bar{d}_k- d_k) + \bar{v}_k -v_k =0.\\
     \end{aligned}
   \end{equation}
   Taking dot product of~\eqref{eqn:sto-convex-dk-diff-1} with $\bar{d}_k-d_k$ and we obtain
   \begin{equation} \label{eqn:sto-convex-dk-diff-2}
     \centering
     \begin{aligned}
	     (\bar{g}_k-g_k)^T(\bar{d}_k-d_k) + \alpha_k \norm{\bar{d}_k- d_k}^2 + (\bar{v}_k-v_k)^T (\bar{d}_k-d_k) =0.\\
     \end{aligned}
   \end{equation}
    Given that $x_k+d_k \in C$ and $x_k+\bar{d}_k \in C$, the convexity of $i_C(\cdot)$ implies 
   \begin{equation} \label{eqn:sto-convex-dk-diff-3}
     \centering
     \begin{aligned}
	     \left(\bar{v}_k -v_k \right)^T(\bar{d}_k-d_k) \geq 0.\\
     \end{aligned}
   \end{equation}
   Applying~\eqref{eqn:sto-convex-dk-diff-3} to~\eqref{eqn:sto-convex-dk-diff-2} leads to 
   \begin{equation} \label{eqn:sto-convex-dk-diff-4}
     \centering
     \begin{aligned}
	     \alpha_k \norm{\bar{d}_k- d_k}^2 &\leq -(\bar{g}_k-g_k)^T(\bar{d}_k-d_k) 
	     \leq \norm{\bar{g}_k-g_k}\norm{\bar{d}_k-d_k}.
     \end{aligned}
   \end{equation}
	Notice that $(\bar{g}_k-g_k)^T(\bar{d}_k-d_k)\leq 0$ from~\eqref{eqn:sto-convex-dk-diff-4} and 
  \begin{equation} \label{eqn:sto-convex-dk-diff-5}
     \centering
     \begin{aligned}
	     \norm{\bar{d}_k-d_k} \leq \frac{1}{\alpha_k} \norm{\bar{g}_k-g_k}.
     \end{aligned}
   \end{equation}
	Taking $\Ebb_k$ on~\eqref{eqn:sto-convex-dk-diff-5} as well as~\eqref{eqn:sto-convex-dk-diff-4} and applying Lemma~\ref{lem:sto-gk-norm} completes the proof.
\end{proof}
The expectation of change in the objective $r(\cdot)$ observes the following Lemma.
\begin{lemma}\label{lem:sto-convex-exp-decrease}
	The step $x_{k+1} = x_k+\bar{d}_k$ satisfies, for constant $c_r>0$, 
  \begin{equation} \label{eqn:sto-convex-exp-0}
    \centering
     \begin{aligned}
	\Ebb_k[r(x_k) - r(x_{k+1})] \geq \frac{1}{2}\Ebb_k[\alpha_k\norm{x_{k+1}-x_k}^2] -c_r M_k. 
    \end{aligned}
    \end{equation}
\end{lemma}
\begin{proof}
	From Lemma~\ref{lem:sto-convex-f-decrease} and~\eqref{eqn:sto-convex-f-decrease-obj}, we have  
\begin{equation} \label{eqn:sto-convex-KKT-0}
 \centering
  \begin{aligned}
	  r(x_k) - r(x_{k+1}) \geq& - g_{k}^T \bar{d}_k -\frac{1}{2}\alpha_k \norm{\bar{d}_k}^2. \\
  \end{aligned}
\end{equation}
	By the optimality conditions~\eqref{eqn:sto-convex-KKT}, we have  
\begin{equation} \label{eqn:sto-convex-KKT-1}
	\bar{g}_k + \alpha_k \bar{d}_k =  -\bar{v}_k, \ \bar{v}_k \in \bar{\partial} i_C(x_k+\bar{d}_k) .
 \end{equation}
	Taking the inner product with $-\bar{d}_k$ in~\eqref{eqn:sto-convex-KKT-1} and using the convexity of $i_C(\cdot)$, we have 
\begin{equation} \label{eqn:sto-convex-KKT-2}
  \centering
   \begin{aligned}
	   - \bar{g}_k^T \bar{d}_k-\alpha_k\norm{\bar{d}_k}^2 &=  \bar{v}_k^T \bar{d}_k
	   = i_C(x_k)-i_C(x_k+\bar{d}_k) - \bar{v}_k^T(-\bar{d}_k)
	   \geq 0.
   \end{aligned}
 \end{equation}
	Thus, 
\begin{equation} \label{eqn:sto-convex-KKT-2.5}
  \centering
   \begin{aligned}
	   - \bar{g}_k^T \bar{d}_k-\frac{1}{2}\alpha_k\norm{\bar{d}_k}^2\geq \frac{1}{2}\alpha_k\norm{\bar{d}_k}^2.
   \end{aligned}
 \end{equation}
	The right-hand side of~\eqref{eqn:sto-convex-KKT-0} can be rewritten based on~\eqref{eqn:sto-convex-KKT-2.5} as
\begin{equation} \label{eqn:sto-convex-KKT-3}
  \centering
   \begin{aligned}
	   -g_k^T \bar{d}_k-\frac{1}{2}\alpha_k \norm{\bar{d}_k}^2&= 
	       -g_k^T \bar{d}_k+\bar{g}_k^T \bar{d}_k-\bar{g}_k^T \bar{d}_k  -\frac{1}{2}\alpha_k \norm{\bar{d}_k}^2\\
	   &\geq \left(\bar{g}_k-g_k \right)^Td_k +\left( \bar{g}_k-g_k\right)^T\left(\bar{d}_k-d_k\right) +\frac{1}{2}\alpha_k \norm{\bar{d}_k}^2 \\
	   &\geq \left(\bar{g}_k-g_k \right)^Td_k - \norm{\bar{g}_k-g_k}\norm{\bar{d}_k-d_k}+\frac{1}{2}\alpha_k\norm{\bar{d}_k}^2\\
	   &\geq \left(\bar{g}_k-g_k \right)^Td_k - \frac{1}{2}(\norm{\bar{g}_k-g_k}^2+\norm{\bar{d}_k-d_k}^2)+\frac{1}{2}\alpha_k\norm{\bar{d}_k}^2.
   \end{aligned}
 \end{equation}
	By Assumption~\ref{assp:gk}, Lemma~\ref{lem:sto-gk-norm} and Lemma~\ref{lem:sto-convex-dk-diff}, the expectation $\Ebb_k$ of~\eqref{eqn:sto-convex-KKT-3} is 
	\begin{equation} \label{eqn:sto-convex-KKT-6}
  \centering
   \begin{aligned}
	   \Ebb_k\left[-g_k^T \bar{d}_k-\frac{1}{2}\alpha_k \norm{\bar{d}_k}^2 \right]
	     \geq& -\frac{1}{2}(1+\frac{1}{\alpha_k^2}) M_k+\Ebb_k\left[\frac{1}{2}\alpha_k\norm{\bar{d}_k}^2\right].
   \end{aligned}
 \end{equation}
   Let $c_r = \frac{1}{2}(1+\frac{1}{\alpha_0^2})$. From~\eqref{eqn:sto-convex-KKT-0} and~\eqref{eqn:sto-convex-KKT-6}, the proof is complete. 
\end{proof}
From Lemma~\ref{lem:sto-convex-exp-decrease}, it is clear that to obtain a convergent step $\bar{d}_k\to 0$, the sequence $M_k$ needs to be controlled so that the right-hand side of~\eqref{eqn:sto-convex-exp-0} is finite in summation. 
Fortunately, this can achieved as standard practice through increasing the sample size $N_k=|S_k|$ as $k\to\infty$.
We present the convergence result in the following theorem.
\begin{theorem}\label{thm:sto-convex-convg}
  If the sequence $M_k$ satisfies    
        $\sum_{k=0}^{\infty} M_k < \infty$,   
       then 
\begin{equation} \label{eqn:sto-convg-limit} 
 \centering
  \begin{aligned}
	  \lim_{k\to\infty} \Ebb\left[\sum_{i=0}^{k-1} \norm{x_{i+1}-x_i}^2\right] < \infty.
  \end{aligned}
\end{equation}
	It follows that $\lim_{k\to\infty} \Ebb\left[\norm{\bar{d}_k}\right] = 0$ and $\lim_{k\to\infty} \Ebb\left[ \norm{\bar{g}_k + \bar{v}_k }  \right] = 0$.
\end{theorem}
\begin{proof}
	Taking the total expectation of~\eqref{eqn:sto-convex-exp-0} in Lemma~\ref{lem:sto-convex-exp-decrease}, for all $k\in \Nbb$,  
  \begin{equation} \label{eqn:sto-convex-convg-1}
   \centering
    \begin{aligned}
	    \Ebb[r(x_{k+1})-r(x_k)] \leq -\frac{1}{2}\Ebb[\alpha_k \norm{\bar{d}_k}^2] + \Ebb\left[c_r M_k\right]
                     \leq -\frac{1}{2}\alpha_0 \Ebb\left[\norm{\bar{d}_k}^2\right] + c_r M_k, \\
    \end{aligned}
  \end{equation}
    since $\alpha_k$ is monotonically non-decreasing and bounded.
    Summing up $i=0,1,\dots,k-1$ of $r(x_{i+1})-r(x_i)$ and taking the total expectation, we have
  \begin{equation} \label{eqn:sto-convex-convg-2}
   \centering
    \begin{aligned}
	    -\infty < r_{\min} - r(x_0) \leq& \Ebb\left[r(x_k) - r(x_0) \right]
	    = \Ebb \left[\sum_{i=0}^{k-1} \left(r(x_{i+1})-r(x_i) \right)   \right]\\
	    \leq& -\frac{1}{2}\alpha_0 \Ebb \left[ \sum_{i=0}^{k-1} \norm{x_{i+1}-x_i}^2 \right] + c_r \sum_{i=0}^{k-1} M_i.\\
    \end{aligned}
  \end{equation}
   Let $b=r(x_0)-r_{\min}$. Then, from the condition of the theorem,  
  \begin{equation} \label{eqn:sto-convex-convg-3}
   \centering
    \begin{aligned}
     \lim_{k\to\infty}\frac{1}{2}  \alpha_0 \Ebb \left[ \sum_{i=0}^{k-1} \norm{x_{i+1}-x_i}^2 \right] \leq  b + c_r \sum_{k=0}^{\infty} M_k <  \infty.
    \end{aligned}
  \end{equation}
	The first part of the theorem is proven.
	Further, $\lim_{k\to\infty} \Ebb[\norm{\bar{d}_k}^2] = 0$. 
        By Jensen's inequality, $\lim_{k\to\infty} \Ebb[\norm{\bar{d}_k}] = 0$.
	By first equation in~\eqref{eqn:sto-convex-KKT}, the last result is proved.
\end{proof} 
It is possible to further show subsequential convergence of accumulation points of $\{x_k\}$.
\begin{theorem}\label{thm:sto-convex-convg-KKT}
	Under the conditions of Theorem~\ref{thm:sto-convex-convg}, any accumulation point of the sequence $\{x_k\}$
	produced by Algorithm~\ref{alg:sto-convex} is a KKT point of~\eqref{eqn:opt-prob-sto-convex} with probability 1.
\end{theorem}
\begin{proof}
	We first show that $\lim_{k\to\infty} \bar{d}_k =0$ with probability 1.
   From Theorem~\ref{thm:sto-convex-convg}, we know 
	    $\sum_{k=0}^{\infty} \Ebb[\norm{\bar{d}_k}^2]$ is finite and $\lim_{k\to\infty} \Ebb[\bar{d}_k] =0$.
    We proceed by contradiction. Suppose there exists $\epsilon>0$ and $a>0$ such that  
  \begin{equation} \label{eqn:sto-convex-convg-KKT-2}
   \centering
    \begin{aligned}
	    \Pbb (\limsup_{k\to\infty} \norm{\bar{d}_k}\geq\epsilon) \geq a. 
    \end{aligned}
  \end{equation}
	By Chebyshev's inequality, we have $\Pbb(\norm{\bar{d}_k}\geq \epsilon) \leq \frac{\Ebb[\norm{\bar{d}_k}^2]}{\epsilon^2}$.
	Since $\Ebb[\norm{\bar{d}_k}^2]$ is finitely summable, there exists $N>0$ such that 
	$\sum_{k=N}^{\infty} \Pbb(\norm{\bar{d}_k}\geq \epsilon) \leq \sum_{k=N}^{\infty}\frac{\Ebb[\norm{\bar{d}_k}^2]}{\epsilon^2}<a$.
	Therefore, 
  \begin{equation} \label{eqn:sto-convex-convg-KKT-3}
   \centering
    \begin{aligned}
	    \Pbb\left( \limsup_{k\to\infty} \norm{\bar{d}_k} \geq \epsilon \right) = \Pbb \left(\limsup_{k\to\infty:k\geq N} \norm{\bar{d}_k}\geq \epsilon\right)
	    \leq \sum_{k=N}^{\infty} \Pbb(\norm{\bar{d}_k}\geq \epsilon) < a.
    \end{aligned}
  \end{equation}
	This is a contradiction against~\eqref{eqn:sto-convex-convg-KKT-2}. Hence, we have 
	$\lim_{k\to\infty} \norm{\bar{d}_k} = 0$ with probability 1. 
        Additionally, under the condition of Theorem~\ref{thm:sto-convex-convg},  
  \begin{equation} \label{eqn:sto-convex-convg-KKT-4}
   \centering
    \begin{aligned}
        \sum_{k=0}^{\infty} \Ebb[\norm{\bar{g}_k-g_k}^2]< \sum_{k=0}^{\infty} M_k < \infty. 
   \end{aligned}
  \end{equation} 
   Using the same contradiction argument, we have $\lim_{k\to\infty} \bar{g}_k-g_k = 0$ with probability 1.
 
Let $\bar{x}$ be an accumulation point of $\{x_k\}$. Then, passing on to a subsequence if necessary, we can assume $\lim_{k\to\infty} x_k = \bar{x}$, where $\bar{x}\in C$. 
   From~\eqref{eqn:sto-convex-KKT}, we have 
  \begin{equation} \label{eqn:sto-convex-convg-KKT-5}
   \centering
    \begin{aligned}
        \bar{g}_k + \alpha_k \bar{d_k} + \bar{v}_k = g_k + (\bar{g}_k -g_k) + \alpha_k \bar{d}_k + \bar{v}_k = 0.
    \end{aligned}
  \end{equation}
Since $g_k\in \bar{\partial}r(x_k)$ is bounded, there exists at least one accumulation point for $\{g_k\}$. Passing on further to a subsequence if necessary, we may assume $g_k\to\bar{g}$. By the outer semicontinuity of Clark subdifferential, we have $\bar{g} \in \bar{\partial} r(\bar{x})$. 
 Therefore,~\eqref{eqn:sto-convex-convg-KKT-5} implies that with probability 1, 
  $\lim_{k\to\infty} \bar{v}_k = -\bar{g}$. 
  Since $\bar{v}_k \in \bar{\partial}i_C(x_{k}+\bar{d}_k)$, by the outer semicontinuity of $\bar{\partial} i_C(\cdot)$, $-\bar{g}\in\bar{\partial} i_C(\bar{x})$.   
   Thus, $0 \in \bar{\partial} r(\bar{x}) + \bar{\partial} i_C(\bar{x})$,
	and $\bar{x}$ is a KKT point with probability 1.
\end{proof}
Similar convergence results can be found in~\cite{liu2020nonsmooth,boob2022}.
The proofs of Theorem~\ref{thm:sto-convex-convg} and~\ref{thm:sto-convex-convg-KKT} can be simplified using the well-known super-martingale convergence theorem~\cite{robbins1971} in the form of the following Lemma.
\begin{lemma}\label{lem:super-martingale-thm}
	Let $(\Omega,\mathcal{F},\Pbb)$ be a probability space and $\mathcal{F}_{k-1}\subset\mathcal{F}_k$ a sequence of sub-$\sigma$-algebra of $\mathcal{F}$. 
	Let $\{y_k\}$,$\{u_k\}$,$\{a_k\}$,$\{b_k\}$ be sequences of nonnegative integrable random variables under $\mathcal{F}_k$, such that for all $k\in\Nbb$, $\Ebb[y_{k+1}|\mathcal{F}_k]\leq (1+a_k)y_k-u_k+b_k$, $\sum_{k=1}^{\infty}a_k<+\infty$, and 
	$\sum_{k=1}^{\infty}b_k<+\infty$. Then with probability 1, $\{y_k\}$ converges and $\sum_{k=1}^{\infty}u_k<+\infty$.
\end{lemma}
In order to apply Lemma~\ref{lem:super-martingale-thm}, we may consider $\mathcal{F}_k$ the $\sigma$-algebra generated by random vectors in $\{ S_t\}_{t=0}^{k}$ and apply the conditional expectation $\Ebb_k$. 

\section{\normalsize Stochastic optimization with equality constraint}\label{sec:alg-line-search}
We now turn to the full problem~\eqref{eqn:opt-prob-sto}.
In addition to Assumptions~\ref{assp:upperC2} and~\ref{assp:boundedHc}, the constraints are assumed to be continuously differentiable with Lipschitz continuous gradient in this section.
\begin{assumption}\label{assp:Lsmoothc}
	The function $c(\cdot)$ is continuously differentiable with Lipschitz continuous gradient. This means that there exists a constant $H\geq 0$ such that 
\begin{equation} \label{eqn:l-smooth-constraint}
 \centering
  \begin{aligned}
	  |c_j(x') - c_j(x) - \nabla c_j(x)^T (x'-x) | \leq&  \frac{H}{2} \norm{x-x'}^2, \\
  \end{aligned}
\end{equation}
	for all $x,x'\in C$ and all $j\in\{1, 2, \ldots, m\}$~\cite{nesterovconvex2003}.
\end{assumption}

\subsection{Algorithm description}\label{sec:alg-eqcons}
The algorithm remains an iterative method with locally approximated convex quadratic objective. 
At iteration $k$ and its iterate $x_k$, the true model $\phi_k(\cdot)$ is the same as 
given in~\eqref{eqn:opt-true-appx} while the stochastic model $\bar{\phi}_k(\cdot)$ is~\eqref{eqn:opt-sto-appx}.
Denoting $d=x-x_k$, $\phi_k(x)$ can be rewritten as $\Phi_k(d)$,  
while its stochastic counterpart is $\bar{\Phi}_k(d)$ in~\eqref{eqn:opt-sto-appx-d}.
Furthermore, the constraint $c(x)=0$ in~\eqref{eqn:opt-prob-sto} is linearized. 
The subproblem to be solved at iteration $k$ is 
\begin{equation} \label{eqn:opt-sto-subproblem}
 \centering
  \begin{aligned}
   &\underset{\substack{x_k+d\in C}}{\text{minimize}} 
	  & & r(x_k)+\bar{g}_k^T d+\frac{1}{2}\alpha_k \norm{d}^2\\
   &\text{subject to}
	  & & c(x_k) + \nabla c(x_k)^T d= 0. 
  \end{aligned}
\end{equation}
The corresponding subproblem with the true subgradient is 
\begin{equation} \label{eqn:opt-true-subproblem}
 \centering
  \begin{aligned}
   &\underset{\substack{x_k+d\in C}}{\text{minimize}} 
	  & & r(x_k)+g_k^T d+\frac{1}{2}\alpha_k \norm{d}^2\\
   &\text{subject to}
	  & & c(x_k) + \nabla c(x_k)^T d= 0. 
  \end{aligned}
\end{equation}
To measure progress in both the objective and constraints, 
a $\ell_1$ merit function is adopted in the form of
\begin{equation} \label{eqn:opt-ms-simp-bundle-merit}
 \centering
  \begin{aligned}
	  \varphi(x,\theta_k) = r(x) + \theta_k \norm{c(x)}_1, \\
  \end{aligned}
\end{equation}
where $\norm{\cdot}_1$ is the 1-norm and $\theta_k>0$ is a penalty parameter. 
A line search step is needed in order to ensure progress
in the merit function~\eqref{eqn:opt-ms-simp-bundle-merit}.
The predicted change $\bar{\delta}_k$ on the objective is again~\eqref{def:pd}. 
Let the line search step size for $\bar{d}_k$ be $\bar{\beta}_k \in (0,1]$. Then, the $(k+1)$-th step taken is given as 
$x_{k+1} = x_k + \bar{\beta}_k \bar{d}_k$.
By  letting $\bar{\delta}_k^{\beta} = \bar{\Phi}_k(0) - \bar{\Phi}_k(\bar{\beta}_k \bar{d}_{k})$, we have  
\begin{equation} \label{def:pd2}
 \centering
  \begin{aligned}
	  \bar{\delta}_k^{\beta} =& \bar{\Phi}_k(0) - \bar{\Phi}_k(\bar{\beta}_k \bar{d}_k)
	   = -\bar{\beta}_k \bar{g}_k^T \bar{d}_k -\frac{1}{2} \alpha_k \bar{\beta}_k^2 \norm{\bar{d}_k}^2.
  \end{aligned}
\end{equation}

The first-order optimality conditions of subproblem~\eqref{eqn:opt-sto-subproblem} are
\begin{equation} \label{eqn:sto-eqcons-KKT}
  \centering
   \begin{aligned}
	   \bar{g}_k + \alpha_k \bar{d}_k + \nabla c(x_k) \bar{\lambda}^{k+1} + \bar{v}_k =0,&\\
	   c(x_k)+\nabla c(x_k)^T\bar{d}_k= 0,&\\
	   x_k,x_k+\bar{d}_k \in C, ~\bar{v}_k \in \bar{\partial} i_C(x_k+\bar{d}_k),&
   \end{aligned}
 \end{equation}
	where $\bar{\lambda}^{k+1}\in \Rbb^m$ is the Lagrange multiplier for the equality constraint.
Similarly, the optimality conditions for~\eqref{eqn:opt-true-subproblem} are
\begin{equation} \label{eqn:true-eqcons-KKT}
  \centering
   \begin{aligned}
	   g_k + \alpha_k d_k + \nabla c(x_k) \lambda^{k+1} + v_k =0,&\\
	   c(x_k)+\nabla c(x_k)^T d_k= 0,&\\
	   x_k,x_k+d_k \in C, ~v_k\in\bar{\partial} i_C(x_k+d_k).&
   \end{aligned}
 \end{equation}
Due to the presence of both equality and inequality constraints, additional step size control is necessary to guarantee 
convergence. Therefore, at $x_k$, the line search condition is presented and 
executed with a scalar $\bar{\zeta}_k \in (0,1]$, and the actual step size is first set to $\bar{\beta}_k = \nu_k \bar{\zeta}_k$. 
The sequence $\{\nu_k\}\subset (0,1]$ is user-defined and deterministic.  
The line search condition with $\bar{d}_k$ and $\bar{\theta}_k$ is
\begin{equation} \label{eqn:line-search-cond-alt}
 \centering
  \begin{aligned}
	  \bar{\theta}_k \norm{c(x_k)}_1 -  \bar{\zeta}_k\left| (\bar{\lambda}^{k+1})^T c(x_k)\right| &\geq \bar{\theta}_k \norm{c(x_{k}+\bar{\zeta}_k\bar{d}_k)}_1-\frac{1}{2} \eta_{\beta} \alpha_k\bar{\zeta}_k\norm{\bar{d}_k}^2. \\
  \end{aligned}
\end{equation}
where $\eta_{\beta}\in (0,1)$ is a parameter of the algorithm. 
In addition, a user-defined upper bound on $\bar{\beta}_k$ needs to be imposed to eventually force $\bar{\beta}_k$ to be sufficiently small while having minimal variance. To that end, by using the ceiling function $\lceil\cdot \rceil$, which returns the least integer greater than the input, we define at iteration $k$ 
\begin{equation} \label{eqn:line-search-pik}
 \centering
  \begin{aligned}
	    \bar{\pi}_k := \min\left\{1,\frac{1}{2} ^{\lceil \log_{\frac{1}{2}} \frac{ \eta_{\beta} \alpha_k}{ H \bar{\theta}_k m} \rceil}\right\}.
  \end{aligned}
\end{equation}
The $\frac{1}{2}$ in definition~\eqref{eqn:line-search-pik} comes from our choice of reduction ratio for $\bar{\zeta}_k$ in step 7 of the algorithm. It can easily be substituted with another number between $0$ and $1$.
We will show later that $\bar{\zeta}_k=\bar{\pi}_k$ satisfies the line search condition in Lemma~\ref{lem:sto-eqcons-line-search}.  
The upper bound on $\bar{\beta}_k$ is set to $\nu_k (\bar{\pi}_k + \mu_k)$, where $\{\mu_k\}\subset[0,1]$ is another user-defined sequence, similar to $\nu_k$.
The parameters $\eta_{\beta}$, $\{\nu_k\}$ and $\{\mu_k\}$ help control the stochastic step size $\bar{\beta}_k$ and its variance to ensure decrease in the merit function value in expectation, which in turn is critical for convergence.
The constant $H$ in~\eqref{eqn:line-search-pik} needs to be estimated for the constraints when it is not known, similar to $\rho$. 

Our proposed stochastic nonsmooth SQP is presented in Algorithm~\ref{alg:sto-eqcons}. 
An adaptive sampling criterion and its implementation is discussed in Section~\ref{sec:adp-sample}. 
\begin{algorithm}
   \caption{Stochastic nonsmooth SQP}\label{alg:sto-eqcons}
	\begin{algorithmic}[1]
		\STATE{
		 Choose scalars $\eta_{\alpha}>1$, $\eta_{\beta}\in (0,1)$,  and $\gamma>0$. Choose $\{\nu_k\}\subset (0,1]$ and $\{\mu_k\}\subset [0,1]$. Initialize $x_0$, $\alpha_0\in [\rho, \eta_{\alpha}\rho]$ and $k=0$.
		 Choose initial sampling set $\{S_0\}$.
		}
	\FOR{$k=0,1,2,...$}
                \STATE{Generate sample set $\{\xi_i\},\xi_i\in S_k$ i.i.d. from probability distribution of $\xi$.}
		\STATE{Call the approximation oracle to obtain $\bar{r}_k$ and subgradient estimate $\bar{g}_k$.}
		\STATE{Form the quadratic function $\bar{\Phi}_k$ in~\eqref{eqn:opt-sto-appx-d} and solve 
		subproblem~\eqref{eqn:opt-sto-subproblem} to obtain $\bar{d}_{k}$ and Lagrange multiplier $\bar{\lambda}^{k+1}$. 
		}
		\STATE{Set $\bar{\theta}_k$ in~\eqref{eqn:opt-ms-simp-bundle-merit} with $\bar{\theta}_k = \max{\{\bar{\theta}_{k-1},\norm{\bar{\lambda}^{k+1}}_{\infty}+\gamma\}}$. }
		  \STATE{Set the initial line search step size $\bar{\zeta}_k=1$. Using backtracking, reducing by half 
		    if too large, find $\bar{\zeta}_k$ such that the condition in~\eqref{eqn:line-search-cond-alt} is satisfied.
		   }
		   \STATE{Set $\bar{\beta}_k =\nu_k \bar{\zeta}_k$. Compute $\bar{\pi}_k$ in~\eqref{eqn:line-search-pik}. Set $\bar{\beta}_k = \min\{\bar{\beta}_k, \nu_k(\bar{\pi}_k+\mu_k) \}$.}
		  \STATE{Take the step $x_{k+1} = x_k+\bar{\beta}_k \bar{d}_k$. }
	    \STATE{Call the chosen $\alpha_k$ update rules to obtain $\alpha_{k+1} \in [\alpha_k,\eta_{\alpha}\rho]$.}
       \ENDFOR
    \end{algorithmic}
\end{algorithm}
As in smooth SQP methods, it is possible that the linearized constraints 
in~\eqref{eqn:opt-sto-subproblem} are infeasible. 
Addressing the inconsistency is beyond the scope of this paper.
Thus, we assume that the solution $\bar{d}_k$ to~\eqref{eqn:opt-sto-subproblem} can be found.
Together with LICQ, we make the following feasibility assumption for the remainder of this section.
\begin{assumption}\label{assp:consistentLICQ}
	The subproblem with the linearized constraints~\eqref{eqn:opt-sto-subproblem} is feasible.
	Further, the constraints $c(x)=0$ and $x\in C$ satisfy LICQ at every accumulation point of $\{x_k\}$ generated by Algorithm~\ref{alg:sto-eqcons}.
\end{assumption}

\subsection{Convergence analysis}\label{sec:sto-eqcons-convg}
The analysis in this section is performed under the Assumptions~\ref{assp:upperC2},~\ref{assp:boundedHc},~\ref{assp:gk},~\ref{assp:Lsmoothc} and~\ref{assp:consistentLICQ}. 
Our choice of $\alpha_k$ in Algorithm~\ref{alg:sto-eqcons} leads to the following Lemma.
\begin{lemma}\label{lem:sto-eqcons-f-decrease}
	The steps $x_{k+1}$ from Algorithm~\ref{alg:sto-eqcons} 
	satisfy $r(x_k)-r(x_{k+1})\geq \Phi(0)-\Phi(\bar{\beta}_k\bar{d}_k)$. 
\end{lemma}
\begin{proof}
	From the upper-$\Ctwo$ inequality~\eqref{eqn:uppc2-def}, we have  
\begin{equation} \label{eqn:sto-eqcons-f-decrease-p1}
 \centering
  \begin{aligned}
	  r(x_k+d)- r(x_k) - g_{k}^Td \leq \frac{\rho}{2} \norm{d}^2,\\
  \end{aligned}
\end{equation}
	for $g_k\in\bar{\partial}r(x_k)$, $\rho>0$ and	$x\in C$.
From Algorithm~\ref{alg:sto-eqcons}, $\alpha_k \geq \rho$, for all  $k\in \Nbb$. 
	The inequality~\eqref{eqn:sto-eqcons-f-decrease-p1} implies that at $d=\bar{\beta}_k\bar{d}_k$
\begin{equation} \label{eqn:sto-eqcons-f-decrease-obj}
 \centering
  \begin{aligned}
	  r(x_k) - r(x_{k+1}) \geq& - \bar{\beta}_k g_{k}^T \bar{d}_k -\frac{1}{2}\rho \bar{\beta}_k^2\norm{\bar{d}_k}^2\\
			\geq& - \bar{\beta}_kg_{k}^T \bar{d}_k -\frac{1}{2}\alpha_k \bar{\beta}_k^2\norm{\bar{d}_k}^2
				 = \Phi_k(0) -\Phi_k(\bar{\beta}_k\bar{d}_k). \\
  \end{aligned}
\end{equation}
\end{proof}

\begin{remark}\label{rmrk:sto-eqcons-alpha}
For simplicity, $\alpha_k$ is required to be monotonically non-decreasing in the algorithm. In practice, experience from the deterministic problems suggests that $\alpha_k$ be reduced when possible to improve convergence behavior. 
	For details see section~\ref{sec:alg-convex-description} of~\cite{wang2022}.
\end{remark}
Next, we show that the line search process is well-defined in the following Lemma.
\begin{lemma}\label{lem:sto-eqcons-line-search}
	If the Lagrange multipliers $\bar{\lambda}^{k+1}$ of~\eqref{eqn:opt-sto-subproblem} are bounded, the line search process of Algorithm~\ref{alg:sto-eqcons} is well-defined. That is,  
	there exists $\bar{\zeta}_k>0$ that satisfies the line search conditions in~\eqref{eqn:line-search-cond-alt} and it can be found in a finite number of backtracking iterations at step 7. Further, $\bar{\beta}_k $ also satisfies
 \begin{equation} \label{eqn:line-search-cond-alt-beta}
 \centering
  \begin{aligned}
	  \bar{\theta}_k \norm{c(x_k)}_1 -  \bar{\beta}_k\left| (\bar{\lambda}^{k+1})^T c(x_k)\right| &\geq \bar{\theta}_k \norm{c(x_{k+1})}_1-\frac{1}{2} \eta_{\beta} \alpha_k\bar{\beta}_k\norm{\bar{d}_k}^2. \\
  \end{aligned}
\end{equation}

\end{lemma}
\begin{proof}
	From Assumption~\ref{assp:Lsmoothc}, we have  
   \begin{equation} \label{eqn:sto-alg-c-ls-pf-1}
   \centering
    \begin{aligned}
	    |c_j(x_k + \bar{\zeta}_k\bar{d}_k)|  \leq & |c_j(x_k)+ \bar{\zeta}_k \nabla c_j(x_k)^T \bar{d}_{k}| + \frac{1}{2}\bar{\zeta}_k^2 H \norm{\bar{d}_k}^2,
   \end{aligned}
   \end{equation}
	$j=1,\dots,m$. 
	Given $\bar{d}_k$ as the solution to~\eqref{eqn:opt-sto-subproblem}, we have that $c_j(x_k) +  \nabla c_j(x_k)^T \bar{d}_k = 0$. As a consequence, we can write based on~\eqref{eqn:sto-alg-c-ls-pf-1} that
   \begin{equation} \label{eqn:sto-alg-c-ls-pf-2}
   \centering
    \begin{aligned}
	    \left|c_j(x_k + \bar{\zeta}_k\bar{d}_k)\right|  \leq \left|(1-\bar{\zeta}_k) c_j(x_k)\right| + \frac{1}{2}\bar{\zeta}_k^2 H \norm{\bar{d}_k}^2,
   \end{aligned}
   \end{equation}
	and thus 
   \begin{equation} \label{eqn:sto-alg-c-ls-pf-3}
   \centering
    \begin{aligned}
	    \norm{c(x_k + \bar{\zeta}_k\bar{d}_k)}_1  \leq 
	     (1-\bar{\zeta}_k) \norm{c(x_k)}_1 + \frac{1}{2}m\bar{\zeta}_k^2 H \norm{\bar{d}_k}^2.
    \end{aligned}
   \end{equation}
On the other hand, simple norm inequalities imply
\begin{equation} \label{eqn:sto-alg-c-ls-pf-4}
   \centering
    \begin{aligned}
	    \bar{\zeta}_k \left|(\bar{\lambda}^{k+1})^T c(x_k)\right| \leq \bar{\zeta}_k\norm{\bar{\lambda}^{k+1}}_{\infty} \norm{c(x_k)}_1.
    \end{aligned}
   \end{equation}
	Since $\bar{\theta}_k \geq \norm{\bar{\lambda}^{k+1}}_{\infty}+\gamma$ from step 6 in Algorithm~\ref{alg:sto-eqcons}, based on~\eqref{eqn:sto-alg-c-ls-pf-3} and~\eqref{eqn:sto-alg-c-ls-pf-4}, 
   \begin{equation} \label{eqn:sto-alg-c-ls-pf-5} 
   \centering
    \begin{aligned}
	    &\bar{\theta}_k\norm{c(x_k)}_1 - \bar{\zeta}_k \left|(\bar{\lambda}^{k+1})^T c(x_k)\right|-\bar{\theta}_k \norm{c(x_k + \bar{\zeta}_k\bar{d}_k)}_1  \\ 
	    &\hspace{1.1cm}\geq(\bar{\theta}_k - \bar{\zeta}_k\norm{\bar{\lambda}^{k+1}}_{\infty}) \norm{c(x_k)}_1 -\bar{\theta}_k(1-\bar{\zeta}_k) \norm{c(x_k)}_1 -\frac{1}{2} \bar{\theta}_k m\bar{\zeta}_k^2 H \norm{\bar{d}_k}^2\\
		&\hspace{1.1cm}=(\bar{\theta}_k\bar{\zeta}_k- \bar{\zeta}_k\norm{\bar{\lambda}^{k+1}}_{\infty}) \norm{c(x_k)}_1-\frac{1}{2} \bar{\theta}_km\bar{\zeta}_k^2 H \norm{\bar{d}_k}^2 \\
		&\hspace{1.1cm}\geq\bar{\zeta}_k \gamma \norm{c(x_k)}_1- \frac{1}{2}\bar{\theta}_k m \bar{\zeta}_k^2H \norm{\bar{d}_k}^2.
    \end{aligned}
   \end{equation}
	Therefore, if $\bar{\zeta}_k$ is reduced through the backtracking of Algorithm~\ref{alg:sto-eqcons} to satisfy
   \begin{equation} \label{eqn:sto-alg-c-ls-pf-6}
   \centering
    \begin{aligned}
	    0< \bar{\zeta}_k \leq \frac{ \eta_{\beta} \alpha_k}{ H \bar{\theta}_k m},
    \end{aligned}
   \end{equation}
then~\eqref{eqn:line-search-cond-alt} is satisfied.
Both the denominator and numerator in~\eqref{eqn:sto-alg-c-ls-pf-6} are positive and independent of the line search. 
If $\bar{\lambda}^{k+1}$ is bounded throughout the algorithm, then $\bar{\theta}_k$ is finite and constant  for $k$ large enough based on how it is chosen in Algorithm~\ref{alg:sto-eqcons}.
	Thus, all terms in~\eqref{eqn:sto-alg-c-ls-pf-6} are finite. Therefore, step 7 of Algorithm~\ref{alg:sto-eqcons} stops in finite iterations. 
    From step 8, $\bar{\beta}_k \leq \min\{\bar{\zeta}_k,\bar{\pi}_k+\mu_k\}$. Hence, it satisfies~\eqref{eqn:sto-alg-c-ls-pf-6} and~\eqref{eqn:line-search-cond-alt-beta}.
\end{proof}
To obtain bounded $\bar{\lambda}^{k+1}$, we first show that the deterministic counterpart $\lambda^{k+1}$ is bounded with LICQ.
\begin{lemma}\label{lem:bounded-lp}
	The sequence of Lagrange multipliers $\{\lambda^{k+1}\}$ for problem~\eqref{eqn:opt-true-subproblem} based on $\{x_k\}$ are bounded. 
	In addition, there exists $\theta_u$ such that $\theta_u \geq \theta_k$ for all $k\in \Nbb$.
\end{lemma}
\begin{proof}
	Note that $\{x_k\}$ is produced by Algorithm~\ref{alg:sto-eqcons}. 
	We rewrite the first equation in optimality condition in~\eqref{eqn:true-eqcons-KKT} as
    \begin{equation} \label{eqn:true-eqcons-KKT-re}
     \centering
     \begin{aligned}
	     g_k + \alpha_k d_k = -\sum_{j=1}^m \lambda^{k+1}_j \nabla c_j(x_k) - v_k .\\
     \end{aligned}
    \end{equation}
	Since $\{x_k\},\{g_k\}$ are bounded ($r$ being Lipschitz on a bounded domain) and $\alpha_k$ is  a finite parameter, 
	the left-hand side of~\eqref{eqn:true-eqcons-KKT-re} stays bounded for all $k\in\Nbb$. 
	Without losing generality, suppose $\lambda^{k+1}_j,j=1,\dots,m$ is not 
	bounded as $k\to\infty$. Since $\{x_k\}$ is bounded, there exists a subsequence $\{x_{k_u}\}$  such that $x_{k_u}\to\bar{x}$. 
	Passing on to a subsequence if necessary, $\lambda^{k_u}_j\to\infty$. 
	If $\bar{x}$ is in the interior of $C$, then $v_{k_u}\to \bar{v} =0$, where $\bar{v} \in \bar{\partial} i_C(\bar{x})$. If $\bar{x}$ is on the boundary of $C$, then the constraints $\bar{x}\in C$ is active. 
	In either case, from Assumption~\ref{assp:consistentLICQ}, $\nabla c_j(\bar{x}), j=1,\cdots,m$ and the non-zero $\bar{v}$ are linearly independent vectors of $\Rbb^n$.
	Then, the right-hand side of~\eqref{eqn:true-eqcons-KKT-re} $-\sum_{j=1}^m \lambda^{k_u+1}_j \nabla c_j(x_{k_u}) -v_{k_u}\to\infty$.
	This is a contradiction. 
	Therefore, there exists $\lambda^u\geq 0$, such that 
	$\norm{\lambda^{k}}_{\infty} \leq \lambda^u$, for all $k$.
	Since $\theta_k$ is determined by $\lambda^k$, there exists $k$ such that $\theta_t=\theta_k$ for all $t\geq k$. 
	Further, since $\theta_k$ is monotonically non-decreasing, there exists $\theta_u>0$ such that $\theta_u\geq\theta_k$ 
	for all $k\in\Nbb$.
\end{proof}

\noindent From Assumption~\ref{assp:gk}, Lemma~\ref{lem:sto-gk-norm} continues to stand in the following form with the same proof. 
\begin{lemma}\label{lem:sto-eqcons-gk-norm}
	For all $k\in \Nbb$ of Algorithm~\ref{alg:sto-eqcons}, $\Ebb_k\left[ \norm{\bar{g}_k-g_k} \right] \leq \sqrt{ M_k} $.
\end{lemma}
The bound on the variance of $\bar{d}_k$ is given next, similar to Lemma~\ref{lem:sto-convex-dk-diff}.
\begin{lemma}\label{lem:sto-eqcons-dk-diff}
	For all $k\in \Nbb$ of Algorithm~\ref{alg:sto-eqcons}, $\Ebb_k[\norm{\bar{d}_k-d_k}] \leq \frac{1}{\alpha_k}\sqrt{M_k}$, and 
	$\Ebb_k[\norm{\bar{d}_k-d_k}^2] \leq \frac{M_k}{\alpha_k^2}$.
\end{lemma}
\begin{proof}
   From the definition, $d_k$ and $\bar{d}_k$ are the solutions to~\eqref{eqn:opt-true-subproblem} and~\eqref{eqn:opt-sto-subproblem}, respectively. 
	Subtracting the first optimality conditions in~\eqref{eqn:true-eqcons-KKT} from~\eqref{eqn:sto-eqcons-KKT} leads to 	
   \begin{equation} \label{eqn:eqcons-dk-diff-2}
     \centering
     \begin{aligned}
	     \bar{g}_k-g_k + \alpha_k(\bar{d}_k- d_k) + \nabla c(x_k)(\bar{\lambda}^{k+1}-\lambda^{k+1})+ \bar{v}_k- v_k =0.\\
     \end{aligned}
   \end{equation}
	Taking the dot product of~\eqref{eqn:eqcons-dk-diff-2} with $\bar{d}_k-d_k$ and we obtain
   \begin{equation} \label{eqn:eqcons-dk-diff-3}
     \centering
     \begin{aligned}
	     (\bar{g}_k-g_k)^T(\bar{d}_k-d_k) + \alpha_k \norm{\bar{d}_k- d_k}^2 &+(\bar{\lambda}^{k+1}-\lambda^{k+1})^T \nabla c(x_k)^T(\bar{d}_k-d_k)\\
	     &+ (\bar{v}_k-v_k)^T(\bar{d}_k-d_k) =0.\\
     \end{aligned}
   \end{equation}
    Given that $x_k+d_k \in C$ and $x_k+\bar{d}_k \in C$, the convexity of $i_C(\cdot)$ implies 
   \begin{equation} \label{eqn:eqcons-dk-diff-4}
     \centering
     \begin{aligned}
	     (\bar{v}_k -v_k )^T(\bar{d}_k-d_k) \geq 0.\\
     \end{aligned}
   \end{equation}
	Further, $\nabla c(x_k)^T \bar{d}_k = \nabla c(x_k)^T d_k = -c(x_k)$, \textit{i.e.},  
   \begin{equation} \label{eqn:eqcons-dk-diff-4.5}
     \centering
     \begin{aligned}
     (\bar{\lambda}^{k+1}-\lambda^{k+1})^T \nabla c(x_k)^T(\bar{d}_k-d_k) = 0.
     \end{aligned}
   \end{equation}
    Applying~\eqref{eqn:eqcons-dk-diff-4} and~\eqref{eqn:eqcons-dk-diff-4.5} to~\eqref{eqn:eqcons-dk-diff-3} leads to 
   \begin{equation} \label{eqn:eqcons-dk-diff-5}
     \centering
     \begin{aligned}
	     \alpha_k \norm{\bar{d}_k- d_k}^2 &\leq -(\bar{g}_k-g_k)^T(\bar{d}_k-d_k) \leq \norm{\bar{g}_k-g_k}\norm{\bar{d}_k-d_k}.
     \end{aligned}
   \end{equation}
   Notice that $(\bar{g}_k-g_k)^T(\bar{d}_k-d_k)\leq 0$.
	Therefore,
  \begin{equation} \label{eqn:eqcons-dk-diff-6}
     \centering
     \begin{aligned}
	     \norm{\bar{d}_k-d_k} \leq \frac{1}{\alpha_k} \norm{\bar{g}_k-g_k}.
     \end{aligned}
   \end{equation}
	Taking $\Ebb_k$ on~\eqref{eqn:eqcons-dk-diff-6} as well as~\eqref{eqn:eqcons-dk-diff-5} and applying Lemma~\ref{lem:sto-eqcons-gk-norm} completes the proof.
\end{proof}
From Lemma~\ref{lem:bounded-lp}, $\{\lambda^k\}$ is bounded. 
To ensure convergence results for the constrained problem, $\{\bar{\lambda}^{k}\}$ also needs to be bounded. 
Unfortunately, this cannot be guaranteed without additional assumptions. 
Choices of assumption in literature include direct boundedness of $\bar{\lambda}^{k}$, boundedness on $\bar{g}_k$~\cite{berahas2021sqpsto} and boundedness on norm of predicted decrease~\cite{berahas2022adaptive}. 
The stochastic subgradient estimate $\bar{g}_k$ is fully determined by the joint distribution of samples $S_0,S_1,\dots,S_k$ through iteration $k$.
By Chebyshev inequality, we know the probability of an unbounded $\bar{g}_k$ can be represented by
 \begin{equation} \label{eqn:chebyshev-1}
  \centering
   \begin{aligned}
	   \Pbb\left( \norm{\bar{g}_k-g_k} \geq a \right) \leq \frac{\Ebb[\norm{\bar{g}_k-g_k}^2]}{a^2},
   \end{aligned}
 \end{equation}
 for any $a>0$.  
 Taking the total expectation of Assumption~\ref{assp:gk} gives us 
 \begin{equation} \label{eqn:chebyshev-2}
  \centering
   \begin{aligned}
	   \Pbb\left( \norm{\bar{g}_k-g_k} \geq a \right) \leq  M_k/ a^2.
   \end{aligned}
 \end{equation}
  Since $M_k$ is bounded, the probability of an unbounded $\bar{g}_k$ tends to zero as $a$ increases and $M_k$ decreases. 
 Given this observation, we assume an essentially bounded $\bar{g}_k$. 

\begin{assumption}\label{assp:sto-gk-essbounded}
 For all $k\in\Nbb$ of Algorithm~\ref{alg:sto-eqcons}, the stochastic subgradient approximation $\bar{g}_k$ is essentially bounded, \textit{i.e.},	 
 \begin{equation} \label{eqn:sto-gk-essbounded-1}
  \centering
   \begin{aligned}
        essup(\bar{g}_k) = \inf \left\{a| \Pbb\{\bar{g}_k | \norm{\bar{g}_k} > a  \} = 0 \right\}
   \end{aligned}
 \end{equation}
	is finite. Equivalently, for all realizations of $\xi$ on the joint distribution of samples $S_0,S_1,\dots,S_k$, $\bar{g}_k$ is bounded with probability 1.
\end{assumption}
We point out Assumption~\ref{assp:sto-gk-essbounded}, or an even more restrictive assumption of a bounded $\bar{g}_k$ can be achieved in practice. 
Suppose $\bar{g}_k$ is given in the form of 
 \begin{equation} \label{eqn:gk-estimate}
  \centering
   \begin{aligned}
        \bar{g}_k = \frac{1}{N_k}\sum_{i=0}^{N_k} G(x_k,\xi_i)
   \end{aligned}
 \end{equation}
  where $G(x_k,\xi_i) \in \bar{\partial} R(x_k, \xi_i)$ and $\Ebb_k [G(x_k,\xi_i)] \in \bar{\partial} r(x_k)$. Then, if $\{R(\cdot,\xi_i)\}$ are Lipschitz, $\bar{g}_k$ is bounded on $C$.
Assumption~\eqref{assp:sto-gk-essbounded} leads to the following lemma.
\begin{lemma}\label{lem:sto-lambda-bounded}
	Given Assumption~\ref{assp:sto-gk-essbounded}, the Lagrange multipliers $\{\bar{\lambda}^k\}$ of~\eqref{eqn:opt-sto-subproblem} are bounded with probability 1. Further, there exists $\bar{\theta}_u>0$ such that $\bar{\theta}_u\geq \bar{\theta}_k$ for all $k\in\Nbb$ with probability 1.
\end{lemma}
\begin{proof}
   From the optimality condition~\eqref{eqn:sto-eqcons-KKT},
    \begin{equation} \label{eqn:sto-eqcons-KKT-re}
     \centering
     \begin{aligned}
	     \bar{g}_k + \alpha_k \bar{d}_k = -\sum_{j=1}^m \bar{\lambda}^{k+1}_j \nabla c_j(x_k) - \bar{v}_k  .\\
     \end{aligned}
    \end{equation}
	The remaining proof is the same as in Lemma~\ref{lem:bounded-lp}, where given Assumption~\ref{assp:sto-gk-essbounded}, the left-hand side of~\eqref{eqn:sto-eqcons-KKT-re} is bounded with probability 1. LICQ thus guarantees that $\{\bar{\lambda}^{k}\}$ and $\bar{\theta}_k$  are bounded with probability 1.
\end{proof}
Denote by $\bar{\theta}_u<\infty$  an upper bound on $\bar{\theta}_k$, the following Lemma on step size follows.
\begin{lemma}\label{lem:sto-eqcons-beta-bounded}
	The step size $\bar{\beta}_k$ satisfies $ \nu_k\bar{\pi}_k \leq \bar{\beta}_k \leq \nu_k(\bar{\pi}_k+\mu_k)$ for all $k\in \Nbb$.
        There exists constants $\pi_l,\pi_u >0 $ so that $\pi_l\leq \bar{\pi}_k\leq \pi_u$ with probability 1. 
        Moreover, there exists $N\in\Nbb$ such that for all $k>N$, $\bar{\pi}_k$ is independent of $\bar{g}_k$ with probability 1, and thus $\Ebb_k[\bar{\pi}_k]=\bar{\pi}_k$, $\nu_k \bar{\pi}_k \leq \Ebb_k[\bar{\beta}_k] \leq \nu_k(\bar{\pi}_k+\mu_k)$.
\end{lemma}
\begin{proof}
	The step size $\bar{\zeta}_k$ is chosen through backtracking, as shown in step 7 and 8 in Algorithm~\ref{alg:sto-eqcons}. 
	From~\eqref{eqn:sto-alg-c-ls-pf-6}, step 7 implies  
   \begin{equation} \label{eqn:sto-beta-bounded-2}
   \centering
    \begin{aligned}
	    \bar{\zeta}_k \geq \min\left\{1,\frac{1}{2} ^{\lceil \log_{\frac{1}{2}} \frac{ \eta_{\beta} \alpha_k}{ H \bar{\theta}_k m} \rceil}\right\}=\bar{\pi}_k.
    \end{aligned}
   \end{equation}
         Therefore, from step 8 in Algorithm~\ref{alg:sto-eqcons},
   \begin{equation} \label{eqn:sto-beta-bounded-3}
   \centering
    \begin{aligned}
          \nu_k \bar{\pi}_k  \leq  \bar{\beta}_k \leq \nu_k (\bar{\pi}_k+\mu_k).
    \end{aligned}
   \end{equation}
        The first part of the Lemma is proven.
        From Algorithm~\ref{alg:sto-eqcons}, $\alpha_k\in [\rho, \eta_{\alpha}\rho]$ is not dependent on stochastic estimate, monotonically non-decreasing and bounded. 
        By Lemma~\ref{lem:sto-lambda-bounded}, $\bar{\theta}_k$ is also  monotonically non-decreasing and bounded with probability 1.
        By definition~\eqref{eqn:line-search-pik}, let 
   \begin{equation} \label{eqn:sto-beta-bounded-1}
   \centering
    \begin{aligned}
               \pi_l  =\min\left\{1,\frac{1}{2} ^{\lceil \log_{\frac{1}{2}} \frac{ \eta_{\beta} \alpha_0}{ H \bar{\theta}_u m} \rceil}\right\}, \ \pi_u =  \min\left\{1,\frac{1}{2} ^{\lceil \log_{\frac{1}{2}} \frac{ \eta_{\beta} \eta_{\alpha}\rho}{ H \bar{\theta}_0 m} \rceil}\right\}.
    \end{aligned}
   \end{equation}
        Then, $\pi_l\leq \bar{\pi}_k\leq \pi_u$ with probability 1. 
        In addition, with probability 1, there exists $N$ such that $\bar{\theta}_k$ is constant for all $k>N$. Thus, $\bar{\pi}_k$ is independent of $\mathcal{F}_k$ for $k>N$ with probability 1. Taking $\Ebb_k$ on~\eqref{eqn:line-search-pik} and~\eqref{eqn:sto-beta-bounded-3} completes the proof.  
\end{proof}
Next, an upper bound is provided for an important term in the analysis of the expectation of merit function.
\begin{lemma}\label{lem:sto-eqcons-exp-decrease-1}
	There exists $c_\beta\geq 0$  such that   
\begin{equation} \label{eqn:sto-eqcons-exp-1-1}
 \centering
  \begin{aligned}
	  \Ebb_k\left[ \bar{\beta}_k (g_k-\bar{g}_k)^T d_k\right] \leq c_{\beta}\nu_k \sqrt{M_k}\norm{d_k}.
  \end{aligned}
\end{equation}
       Further, there exists $N$, such that for $k>N$,  
\begin{equation} \label{eqn:sto-eqcons-exp-1-1.5}
 \centering
  \begin{aligned}
	  \Ebb_k\left[ \bar{\beta}_k (g_k-\bar{g}_k)^T d_k\right] \leq \nu_k\mu_k \sqrt{M_k}\norm{d_k}.
  \end{aligned}
\end{equation}
\end{lemma}
\begin{proof}
	Let $\Pbb_k[\cdot]$ denote the probability conditioned on the event of $x_k$ at $k$th iteration. 
        From Lemma~\ref{lem:sto-eqcons-beta-bounded}, we know that $\nu_k\pi_l\leq\bar{\beta}_k\leq \nu_k(\pi_u+\mu_k)$ for all $k$ with probability 1, which by the Law of Total Expectation is the only case we need to consider when taking $\Ebb_k$. 
	Let $A_k$ be the event that $(g_k-\bar{g}_k)^Td_k\geq 0$ and $A_k^c$ the event that $(g_k-\bar{g}_k)^Td_k < 0$. 
	Using the Law of Total Expectation, we have 
  \begin{equation} \label{eqn:sto-eqcons-exp-1-2}
   \centering
    \begin{aligned}
	    &\Ebb_k\left[ \bar{\beta}_k (g_k-\bar{g}_k)^T d_k\right] \\
	    =& \Ebb_k\left[ \bar{\beta}_k (g_k-\bar{g}_k)^T d_k|A_k  \right]\Pbb_k[A_k] +\Ebb_k\left[ \bar{\beta}_k (g_k-\bar{g}_k)^T d_k|A_k^c  \right]\Pbb_k[A_k^c]\\
	    \leq& \nu_k (\pi_u+\mu_k)  \Ebb_k\left[(g_k-\bar{g}_k)^T d_k|A_k  \right]\Pbb_k[A_k] +\nu_k \pi_l\Ebb_k\left[(g_k-\bar{g}_k)^T d_k|A_k^c  \right]\Pbb_k[A_k^c]
    \end{aligned}
  \end{equation}
	Since $\Ebb_k[\bar{g}_k]=g_k$,~\eqref{eqn:sto-eqcons-exp-1-2} implies
  \begin{equation} \label{eqn:sto-eqcons-exp-1-3}
   \centering
    \begin{aligned}
	    \Ebb_k [ \bar{\beta}_k (g_k-\bar{g}_k)^T d_k]
	    \leq& \nu_k \pi_l \Ebb_k\left[(g_k-\bar{g}_k)^T d_k|A_k  \right]\Pbb_k[A_k] +\\
	    \nu_k \pi_l \Ebb_k [(g_k-\bar{g}_k)^T &d_k|A_k^c  ]\Pbb_k[A_k^c]+
	    \nu_k(\pi_u+\mu_k-\pi_l) \Ebb_k\left[(g_k-\bar{g}_k)^T d_k|A_k  \right]\Pbb_k[A_k]\\
	    =&\nu_k (\pi_u+\mu_k-\pi_l) \Ebb_k\left[(g_k-\bar{g}_k)^T d_k|A_k  \right]\Pbb_k[A_k].
    \end{aligned}
  \end{equation}
	Similarly, 
  \begin{equation} \label{eqn:sto-eqcons-exp-1-4}
   \centering
    \begin{aligned}
	   \Ebb_k\left[ \bar{\beta}_k (g_k-\bar{g}_k)^T d_k\right] 
	    \leq&\nu_k(\pi_l-\mu_k-\pi_u) \Ebb_k\left[(g_k-\bar{g}_k)^T d_k|A_k^c\right]\Pbb_k[A_k^c].
    \end{aligned}
  \end{equation}
	Summing~\eqref{eqn:sto-eqcons-exp-1-3} and~\eqref{eqn:sto-eqcons-exp-1-4} leads to 
  \begin{equation} \label{eqn:sto-eqcons-exp-1-5}
   \centering
    \begin{aligned}
	    &\Ebb_k\left[ \bar{\beta}_k (g_k-\bar{g}_k)^T d_k\right] \\
	    \leq&\frac{1}{2}\nu_k (\pi_u+\mu_k-\pi_l)\left( \Ebb_k\left[(g_k-\bar{g}_k)^T d_k|A_k  \right]\Pbb_k[A_k]-\Ebb_k\left[(g_k-\bar{g}_k)^T d_k|A_k^c\right]\Pbb_k[A_k^c]\right).
    \end{aligned}
  \end{equation}
       Using the Law of Total Expectation again, we have 
  \begin{equation} \label{eqn:sto-eqcons-exp-1-6}
   \centering
    \begin{aligned}
	    \Ebb_k\left[(g_k-\bar{g}_k)^T d_k|A_k  \right]\Pbb_k[A_k] &\leq \Ebb_k \left[\norm{g_k-\bar{g}_k}\norm{d_k}|A_k\right]\Pbb_k[A_k]\\
	    &=\Ebb_k \left[\norm{g_k-\bar{g}_k}\norm{d_k}\right] -\Ebb_k \left[\norm{g_k-\bar{g}_k}\norm{d_k}|A_k^c\right]\Pbb_k[A_k^c]\\
	    &\leq \Ebb_k [\norm{g_k-\bar{g}_k}] \norm{d_k}.
    \end{aligned}
  \end{equation}
	Similarly, 
  \begin{equation} \label{eqn:sto-eqcons-exp-1-7}
   \centering
    \begin{aligned}
	    -\Ebb_k\left[(g_k-\bar{g}_k)^T d_k|A_k^c\right]\Pbb_k[A_k^c] &\leq \Ebb_k \left[\norm{g_k-\bar{g}_k}\norm{d_k}|A_k^c\right]\Pbb_k[A_k^c]\\
	    &=\Ebb_k \left[\norm{g_k-\bar{g}_k}\norm{d_k}\right] -\Ebb_k \left[\norm{g_k-\bar{g}_k}\norm{d_k}|A_k\right]\Pbb_k[A_k]\\
	    &\leq \Ebb_k [\norm{g_k-\bar{g}_k}] \norm{d_k}.
    \end{aligned}
  \end{equation}
	Applying~\eqref{eqn:sto-eqcons-exp-1-6} and~\eqref{eqn:sto-eqcons-exp-1-7} to~\eqref{eqn:sto-eqcons-exp-1-5}, we have 
  \begin{equation} \label{eqn:sto-eqcons-exp-1-8}
   \centering
    \begin{aligned}
	    \Ebb_k\left[ \bar{\beta}_k (g_k-\bar{g}_k)^T d_k\right] 
	    \leq& \nu_k (\pi_u+\mu_k-\pi_l) \Ebb_k [\norm{g_k-\bar{g}_k}] \norm{d_k}.
    \end{aligned}
  \end{equation}
   By Lemma~\ref{lem:sto-eqcons-gk-norm}, let $c_{\beta}=\pi_u-\pi_l\geq \pi_u+\mu_k-\pi_l\geq 0$ for all $k$,~\eqref{eqn:sto-eqcons-exp-1-1} is proven.
   
   From Lemma~\ref{lem:sto-eqcons-beta-bounded}, there exists $N$, such that we can replace $\pi_u$, $\pi_l$ with $\bar{\pi}_k$ in~\eqref{eqn:sto-eqcons-exp-1-2},~\eqref{eqn:sto-eqcons-exp-1-3},~\eqref{eqn:sto-eqcons-exp-1-4},~\eqref{eqn:sto-eqcons-exp-1-5} for $k>N$.
   Then, from~\eqref{eqn:sto-eqcons-exp-1-8},~\eqref{eqn:sto-eqcons-exp-1-1.5} is proven.
\end{proof}
We proceed to examine the expectation $\Ebb_k$ of consecutive merit function values.
\begin{lemma}\label{lem:sto-eqcons-exp-decrease}
	There exists constants $c_1,c_2,c_d>0$, $N\in\Nbb$ and sequence $\{c_k^{\mu}\}$ so that the step $x_{k+1} = x_k+\bar{\beta}_k \bar{d}_k$ leads to  
\begin{equation} \label{eqn:sto-eqcons-exp-decrease-0}
 \centering
  \begin{aligned}
	\Ebb_k\left[\varphi(x_k,\bar{\theta}_k)-\varphi(x_{k+1},\bar{\theta}_k)\right]\geq c_1 \nu_k \Ebb_k [\norm{\bar{d}_k}^2] - c_d c_k^{\mu}\nu_k\sqrt{M_k}-c_2\nu_k M_k,
  \end{aligned}
\end{equation}
      where $c_d$ is the upper bound such that $\norm{d_k}\leq c_d$, and $c_k^{\mu}=c_{\beta}$ for $k\leq N$ and $c_k^{\mu}=\mu_k$ for $k>N$. 
\end{lemma}
\begin{proof}
	From Lemma~\ref{lem:sto-eqcons-f-decrease}, we have  
\begin{equation} \label{eqn:sto-eqcons-KKT-1}
 \centering
  \begin{aligned}
	  r(x_k) - r(x_{k+1}) \geq& -\bar{\beta}_k g_{k}^T\bar{d}_k -\frac{1}{2}\alpha_k\bar{\beta}_k^2 \norm{\bar{d}_k}^2. 
  \end{aligned}
\end{equation}
		Rearranging the first equation in optimality conditions~\eqref{eqn:sto-eqcons-KKT}, we have  
\begin{equation} \label{eqn:sto-eqcons-KKT-beta}
	\bar{g}_k + \alpha_k \bar{d}_k =  -\nabla c(x_k) \bar{\lambda}^{k+1} - \bar{v}_k .
 \end{equation}
	Taking the inner product with $-\bar{d}_k$ in~\eqref{eqn:sto-eqcons-KKT-beta} and using~\eqref{eqn:sto-eqcons-KKT}, we have 
\begin{equation} \label{eqn:sto-eqcons-KKT-2}
  \centering
   \begin{aligned}
	   - \bar{g}_k^T \bar{d}_k-\alpha_k\norm{\bar{d}_k}^2 &= (\bar{\lambda}^{k+1})^T \nabla c(x_k)^T \bar{d}_k+ \bar{v}_k\bar{d}_k\\
			    &= -(\bar{\lambda}^{k+1})^T c(x_k) + i_C(x_k) -i_C(x_k+\bar{d}_k) - \bar{v}_k^T (-\bar{d}_k)\\
			    &\geq -(\bar{\lambda}^{k+1})^T c(x_k), 
   \end{aligned}
 \end{equation}	
	where the convexity of $i_C(\cdot)$ is applied.
	Next, multiplying both sides of~\eqref{eqn:sto-eqcons-KKT-2} by $\bar{\beta}_k$ and then subtracting $\frac{1}{2}\alpha_k\bar{\beta}_k^2\norm{\bar{d}_k}^2$ leads to
  \begin{equation} \label{eqn:sto-eqcons-KKT-3}
       \centering
       \begin{aligned}
	       -\bar{\beta}_k\bar{g}_k^T \bar{d}_k-\frac{1}{2}\alpha_k\bar{\beta}_k^2\norm{\bar{d}_k}^2 &\geq \alpha_k\bar{\beta}_k\norm{\bar{d}_k}^2-\frac{1}{2}\alpha_k\bar{\beta}_k^2\norm{\bar{d}_k}^2-\bar{\beta}_k(\bar{\lambda}^{k+1})^Tc(x_k)\\
		 &\geq\frac{1}{2}\alpha_k\bar{\beta}_k\norm{\bar{d}_k}^2-\bar{\beta}_k(\bar{\lambda}^{k+1})^Tc(x_k),
         \end{aligned}
        \end{equation}
	where the second inequality makes use of $\bar{\beta}_k\in (0,1]$.
The right-hand side of~\eqref{eqn:sto-eqcons-KKT-1} can be rewritten through that of~\eqref{eqn:sto-eqcons-KKT-3} as
\begin{equation} \label{eqn:sto-eqcons-KKT-4}
  \centering
   \begin{aligned}
	   -&\bar{\beta}_k g_k^T \bar{d}_k-\frac{1}{2}\alpha_k \bar{\beta}_k^2 \norm{\bar{d}_k}^2= 
	       -\bar{\beta}_k g_k^T \bar{d}_k+\bar{\beta}_k \bar{g}_k^T \bar{d}_k-\bar{\beta}_k \bar{g}_k^T \bar{d}_k  -\frac{1}{2}\alpha_k\bar{\beta}_k^2 \norm{\bar{d}_k}^2\\
	   &\geq \bar{\beta}_k\left(\bar{g}_k-g_k \right)^Td_k +\bar{\beta}_k \left(\bar{g}_k-g_k\right)^T\left(\bar{d}_k-d_k\right) +\frac{1}{2}\alpha_k \bar{\beta}_k \norm{\bar{d}_k}^2 -\bar{\beta}_k(\bar{\lambda}^{k+1})^Tc(x_k)\\
	   &\geq  \bar{\beta}_k\left(\bar{g}_k-g_k \right)^Td_k  -\bar{\beta}_k \norm{\bar{g}_k-g_k}\norm{\bar{d}_k-d_k}+\frac{1}{2}\alpha_k\bar{\beta}_k \norm{\bar{d}_k}^2-\bar{\beta}_k(\bar{\lambda}^{k+1})^Tc(x_k).
   \end{aligned}
 \end{equation}
From Lemma~\ref{lem:sto-eqcons-line-search} and~\eqref{eqn:line-search-cond-alt-beta}, we have  
\begin{equation} \label{eqn:sto-eqcons-KKT-4.5}
  \centering
   \begin{aligned}
	   -\bar{\beta}_k(\bar{\lambda}^{k+1})^Tc(x_k)+\bar{\theta}_k\norm{c(x_k)}_1-\bar{\theta}_k\norm{c(x_{k+1})}_1
	   \geq -\frac{1}{2}\eta_{\beta} \alpha_k \bar{\beta}_k \norm{\bar{d}_k}^2.
   \end{aligned}
 \end{equation}
Combine the inequalities in~\eqref{eqn:sto-eqcons-KKT-1},~\eqref{eqn:sto-eqcons-KKT-4} and~\eqref{eqn:sto-eqcons-KKT-4.5}, we have  
\begin{equation} \label{eqn:sto-eqcons-KKT-5}
 \centering
  \begin{aligned}
	  \varphi(x_k,\bar{\theta}_k) -& \varphi(x_{k+1},\bar{\theta}_k)  = r(x_k) -r(x_{k+1})+\bar{\theta}_k\norm{c(x_k)}_1-\bar{\theta}_k\norm{c(x_{k+1})}_1\\
			 \geq& -\bar{\beta}_kg_k^T \bar{d}_k-\frac{1}{2}\alpha_k\bar{\beta}_k^2\norm{\bar{d}_k}^2+\bar{\theta}_k\norm{c(x_k)}_1-\bar{\theta}_k\norm{c(x_{k+1})}_1 \\
			 \geq& \bar{\beta}_k (\bar{g}_k-g_k)^T d_k -\bar{\beta}_k \norm{\bar{g}_k-g_k}\norm{\bar{d}_k-d_k}+ (1-\eta_{\beta})\frac{1}{2}\alpha_k\bar{\beta}_k\norm{\bar{d}_k}^2.
  \end{aligned}
\end{equation}
	From Lemma~\ref{lem:sto-eqcons-exp-decrease-1}, there exists $N\in\Nbb$ that we can define $\{c_k^{\mu}\}$ as stated in this lemma, and have 
\begin{equation} \label{eqn:sto-eqcons-KKT-5.5}
 \centering
  \begin{aligned}
	  \Ebb_k[\bar{\beta}_k (\bar{g}_k-g_k)^T d_k]\geq  -\nu_k c_k^{\mu} \sqrt{M_k} \norm{d_k}.
  \end{aligned}
\end{equation}
Notice that $\mu_k\in [0,1]$. From Lemma~\ref{lem:sto-eqcons-beta-bounded}, using the Law of Total Expectation and applying Lemma~\ref{lem:sto-eqcons-dk-diff} and~\ref{lem:sto-eqcons-gk-norm}, we have 
\begin{equation} \label{eqn:sto-eqcons-KKT-8}
\centering
 \begin{aligned}
	 \Ebb_k&\left[- \bar{\beta}_k\norm{\bar{g}_k-g_k}\norm{\bar{d}_k-d_k}\right] \geq
	 - \nu_k (\pi_u+1) \Ebb_k[\norm{\bar{g}_k-g_k}\norm{\bar{d}_k-d_k}]\\ 
          \geq& - \nu_k (\pi_u+1) \frac{1}{2}\Ebb_k[\norm{\bar{g}_k-g_k}^2 + \norm{\bar{d}_k-d_k}^2]\geq 
         - \frac{\nu_k}{2} (\pi_u+1) (M_k + \frac{M_k}{\alpha_k^2}).
 \end{aligned}
\end{equation}
From~\eqref{eqn:sto-eqcons-KKT-5.5} and~\eqref{eqn:sto-eqcons-KKT-8}, the expectation $\Ebb_k$ on both sides of~\eqref{eqn:sto-eqcons-KKT-5} is 
\begin{equation} \label{eqn:sto-eqcons-KKT-9}
 \centering
  \begin{aligned}
	  \Ebb_k\left[\varphi(x_k,\bar{\theta}_k) - \varphi(x_{k+1},\bar{\theta}_k)\right]  \geq  (1-\eta_{\beta})\frac{\alpha_k}{2}\nu_k \pi_l\Ebb_k[\norm{\bar{d}_k}^2 ]- c_d \nu_k c_k^{\mu}\sqrt{M_k} - c_2 \nu_k M_k,
  \end{aligned}
\end{equation}
where $c_2=\frac{\alpha_0^2+1}{2\alpha_0^2}(\pi_u+1)$.  Let $c_1 = (1-\eta_{\beta})\frac{\alpha_0}{2} \pi_l$.
Then,~\eqref{eqn:sto-eqcons-KKT-9} becomes~\eqref{eqn:sto-eqcons-exp-decrease-0}.
\end{proof}
From Lemma~\ref{lem:sto-eqcons-exp-decrease}, for $k$ large enough $c_k^{\mu}$ is replaced by a user-defined sequence $\{\mu_k\}$, which gives us more options to obtain convergence.
Let $M_c>0$ be the upper bound of the constraints so that $\norm{c(x)}_1\leq M_c$ for all $x\in C$.
The expected value of the merit function follows in the next lemma.
\begin{lemma}\label{lemma:sto-eqcons-merit}
        There exists constant $M_{\varphi}>0$, so that  
\begin{equation} \label{eqn:sto-eqcons-limit-1}
 \centering
  \begin{aligned}
           \Ebb \left[c_1 \sum_{i=0}^{k-1} \nu_i \norm{\bar{d}_i}^2\right]  \leq M_{\varphi} +  c_d \sum_{i=0}^{k-1} \nu_i c_i^{\mu} \sqrt{M_i} + c_2 \sum_{i=0}^{k-1} \nu_i M_i.
  \end{aligned}
\end{equation}
\end{lemma}
\begin{proof}
	By Lemma~\ref{lem:sto-eqcons-exp-decrease}, we have  
  \begin{equation} \label{eqn:sto-eqcons-convg-1}
   \centering
    \begin{aligned}
	    \Ebb_k[\varphi(x_{k+1},\bar{\theta}_k)-\varphi(x_k,\bar{\theta}_k)] \leq -c_1\nu_k \Ebb_k [\norm{\bar{d}_k}^2]+c_d c_k^{\mu}\nu_k\sqrt{M_k}+c_2\nu_k M_k . \\
    \end{aligned}
  \end{equation}
	Since both $r(\cdot)$ and $\norm{c(\cdot)}_1$ are bounded below, so is $\varphi(\cdot,\cdot)$. Let $\varphi_{m}$ be 
	the minimum of $\varphi(\cdot,\cdot)$ on $C$. Denote $\varphi_{0} = \varphi(x_0,\bar{\theta}_0)$ for brevity. Summing up $i=0,1,\dots,k-1$ of $\varphi(x_{i+1},\bar{\theta}_{i+1})-\varphi(x_i,\bar{\theta}_i)$ and taking the total expectation, we have
  \begin{equation} \label{eqn:sto-eqcons-convg-2}
   \centering
    \begin{aligned}
	    -\infty &< \varphi_{m} - \varphi_{0} \leq \Ebb\left[\varphi(x_k,\bar{\theta}_k) - \varphi_{0} \right]
	    = \Ebb \left[\sum_{i=0}^{k-1} \left(\varphi(x_{i+1},\bar{\theta}_{i+1})-\varphi(x_i,\bar{\theta}_i) \right)  \right]\\
	    \leq& \Ebb \left[ \sum_{i=0}^{k-1} (\bar{\theta}_{i+1}-\bar{\theta}_i) \norm{c(x_{i+1})}_1  - c_1\sum_{i=0}^{k-1} \nu_i \norm{\bar{d}_i}^2 + c_d\sum_{i=0}^{k-1} c_i^{\mu} \nu_i  \sqrt{M_i} + c_2\sum_{i=0}^{k-1}  \nu_i M_i\right].\\
    \end{aligned}
  \end{equation}
	Therefore, 
  \begin{equation} \label{eqn:sto-eqcons-convg-3}
   \centering
    \begin{aligned}
           \Ebb \left[c_1 \sum_{i=0}^{k-1} \nu_i \norm{\bar{d}_i}^2\right]   \leq \varphi_{0} - \varphi_{m}+ (\bar{\theta}_u-\bar{\theta}_0)M_c 
                + c_d\sum_{i=0}^{k-1} \nu_i c_i^{\mu}  \sqrt{M_i} + c_2 \sum_{i=0}^{k-1} \nu_i M_i.
    \end{aligned}
  \end{equation}
     Let $M_{\varphi} = \varphi_{0}-\varphi_{m}+(\bar{\theta}_u-\bar{\theta}_0)M_c$,~\eqref{eqn:sto-eqcons-limit-1} is proven. 
\end{proof}
From Lemma~\ref{lemma:sto-eqcons-merit}, to obtain a convergent step $\bar{d}_k\to 0$, the sequences $\{\nu_k\},\{\mu_k\}$ and $\{M_k\}$ need to be controlled so that the right hand side of~\eqref{eqn:sto-eqcons-limit-1} is finite in summation, while $\sum_{i=1}^{k-1}\nu_i$ is not finite as $k\to\infty$. 
There exist a number of combination of conditions to ensure such is the case. However, in any case, an approach purely based on $\nu_k$ would not suffice, as $\nu_k$ is present on both sides of the inequality. The stochastic subgradient variance $M_k$  needs to be reduced as $k$ increases.
This can be achieved as standard practice through increasing the sample size $N_k=|S_k|$. Moreover, $\mu_k$ can be reduced as part of the algorithm, which reduces the variance of step size $\bar{\beta}_k$ due to stochastic estimate. 
We present one convergence result in the following theorem.
\begin{theorem}\label{thm:sto-eqcons-convg-2}
	If the sequences $\nu_k$, $\mu_k$ and $M_k$ satisfy 
  \begin{equation} \label{eqn:sto-eqcons-nu-M-2}
   \centering
    \begin{aligned}
         \limsup_{k\to\infty} \nu_k > 0, \ \  \sum_{k=0}^{\infty} \mu_k \sqrt{M_k} < \infty,   \ \  \sum_{k=0}^{\infty} M_k < \infty,  
    \end{aligned}
  \end{equation}
       then  
  \begin{equation} \label{eqn:sto-eqcons-expectation-2}
   \centering
    \begin{aligned}
         \lim_{k\to\infty} \Ebb\left[ \sum_{i=0}^{k} \norm{\bar{d}_i}^2\right]  < \infty , \ \lim_{k\to\infty} \Ebb\left[\norm{\bar{d}_k}\right] = 0.
    \end{aligned}
  \end{equation}
	Further, every accumulation point of the sequence $\{x_k\}$
	produced by Algorithm~\ref{alg:sto-eqcons} is a KKT point of~\eqref{eqn:opt-prob-sto} with probability 1
\end{theorem}
\begin{proof}
  From~\eqref{eqn:sto-eqcons-nu-M-2} and $\nu_k \subset (0,1]$, there exists $c_{\nu}>0$ such that $\nu_k>c_{\nu}$ for all $k$. Thus, by~\eqref{eqn:sto-eqcons-limit-1},
   \begin{equation} \label{eqn:sto-eqcons-expectation-pf-1}
   \centering
    \begin{aligned}
           \Ebb \left[c_1c_{\nu} \sum_{i=0}^{k-1}  \norm{\bar{d}_i}^2\right]\leq \Ebb \left[c_1 \sum_{i=0}^{k-1} \nu_i \norm{\bar{d}_i}^2\right]   < M_{\varphi} +  c_d \sum_{i=0}^{k-1} c_i^{\mu} \sqrt{M_i} + c_2 \sum_{i=0}^{k-1} M_i.
    \end{aligned}
  \end{equation}
    Take $k\to\infty$, by the statement of Lemma~\ref{lem:sto-eqcons-exp-decrease} and~\eqref{eqn:sto-eqcons-nu-M-2},~\eqref{eqn:sto-eqcons-expectation-2} is obtained.
        It follows that $\lim_{k\to\infty} \Ebb[\norm{\bar{d}_k}^2] = 0$. First part of the theorem is proven.

     Next, using the same contradiction argument used in Theorem~\ref{thm:sto-convex-convg-KKT} on~\eqref{eqn:sto-eqcons-expectation-2}, 
     we have $\lim_{k\to\infty} \norm{\bar{d}_k}^2 = 0$ with probability 1. Therefore, $\lim_{k\to\infty}\bar{d}_k =0$ with probability 1.   By~\eqref{eqn:sto-eqcons-KKT}, $\lim_{k\to\infty} c(x_k) = 0$ with probability 1.
    Additionally, by Assumption~\ref{assp:gk} and~\eqref{eqn:sto-eqcons-nu-M-2}, 
  \begin{equation} \label{eqn:sto-eqcons-convg-KKT-1}
   \centering
    \begin{aligned}
        \sum_{k=0}^{\infty} \Ebb[\norm{\bar{g}_k-g_k}^2]  < \infty. 
   \end{aligned}
  \end{equation} 
   Using the same contradiction argument again, we have $\lim_{k\to\infty} \bar{g}_k-g_k = 0$ with probability 1. 

      Let $\bar{x}$ be an accumulation point of $\{x_k\}$. Then, passing on to a subsequence if necessary, we can assume $\lim_{k\to\infty} x_k = \bar{x}$ where $\bar{x}\in C$. By Lemma~\ref{lem:sto-lambda-bounded}, $\{\bar{\lambda}^k\}$ is bounded with probability 1. 
        Further, $g_k\in \bar{\partial}r(x_k)$ is bounded. 
	Thus, there exist accumulation points for $\{g_k\}$ and $\{\bar{\lambda}^k\}$ with probability 1. 
        Passing on to a subsequence if necessary, we assume $g_k\to\bar{g}$ and $\bar{\lambda}^k\to\bar{\lambda}$ with probability 1. 
        By the outer semicontinuity of Clark subdifferential, we have $\bar{g} \in \bar{\partial} r(\bar{x})$.
	From~\eqref{eqn:sto-eqcons-KKT},  
\begin{equation} \label{eqn:sto-eqcons-convg-KKT-2}
   \centering
    \begin{aligned}
        \bar{g}_k + \alpha_k \bar{d_k} + \nabla c(x_k)\bar{\lambda}^{k+1} + \bar{v}_k = g_k + (\bar{g}_k -g_k) + \alpha_k \bar{d}_k + \nabla c(x_k)\bar{\lambda}^{k+1}+\bar{v}_k = 0.
    \end{aligned}
  \end{equation}
         Thus, $\lim_{k\to\infty}\bar{v}_k = -\bar{g}-\nabla c(\bar{x})\bar{\lambda}$  with probability 1. Given that $\bar{v}_k \in \bar{\partial} i_C(x_k+\bar{d}_k)$, the outer semicontinuity of $\bar{\partial} i_C(\cdot)$ leads to $ \lim_{k\to\infty} \bar{v}_k \in \bar{\partial} i_C(\bar{x})$. Therefore, with probability 1,
	    $0 \in \bar{\partial} r(\bar{x}) + \nabla c(\bar{x})\bar{\lambda} + \bar{\partial} i_C(\bar{x})$.
	Thus, $\bar{x}$ is a KKT point with probability 1.
\end{proof}

\section{\normalsize Adaptive Sampling algorithm}\label{sec:adp-sample}
In this section, we present an adaptive sampling strategy for Algorithm~\ref{alg:sto-convex} and~\ref{alg:sto-eqcons} that determines sequence $\{M_k\}$.
Specifically, instead of Assumption~\ref{assp:gk}, the following assumption is made. 
 \begin{assumption}\label{assp:gk-adaptive}
   For all iterations $k\in \Nbb$, the stochastic subgradient approximation $\bar{g}_k$ is an
    unbiased estimate of $g_k\in\bar{\partial} r(x_k)$, \textit{i.e.}, $\Ebb_k[\bar{g}_k]=g_k$.
    Furthermore, there exists $\eta>0$, such that $\bar{g}_k$ satisfies
    \begin{equation} \label{assp:sto-grad-variance-adaptive}
     \centering
      \begin{aligned}
          \Ebb_k\left[\norm{\bar{g}_k-g_k}^2 \right]\leq M_k \leq \eta \alpha_k \norm{d_k}^2.
      \end{aligned}
    \end{equation}
\end{assumption}

\begin{remark}
   Assumption~\ref{assp:gk-adaptive} reduces to the well-known norm condition for unconstrained optimization problems.
We point out that it is not uncommon for adaptive sampling analysis to rely on theoretical conditions such as~\eqref{assp:sto-grad-variance-adaptive} that are not implementable. While a practical implementation of~\eqref{assp:sto-grad-variance-adaptive} will be given later in the section, there exists a gap between the convergence guarantee and implementation, though these methods have been shown to enjoy success in applications~\cite{berahas2022adaptive,bollapragada2018adaptive}. Recently, some algorithms have incorporated the trajectory-dependent adaptive sampling and its biased expected value into the analysis~\cite{shashaani2018astro}, a topic for future work of the authors. 
\end{remark}
An adaptive sample size  has the potential to increase algorithm efficiency and reduce iterations needed for convergence. 
From section~\ref{sec:sto-eqcons-convg}, $M_k\to 0$ is necessary for convergence. Therefore, the adaptive sampling criterion needs to employ quantities that tend to $0$ if the algorithm converges. In addition to $\norm{d_k}$, reduction in value of merit function $\varphi(\cdot,\cdot)$, predicted change in model value $\Phi_k(0)-\Phi_k(d_k)$ can be considered for the right-hand side of the inequality in~\eqref{assp:gk-adaptive}. 
We present the convergence analysis under Assumption~\ref{assp:gk-adaptive}, based on the results from section~\ref{sec:sto-eqcons-convg}. Given Assumptions~\ref{assp:upperC2},~\ref{assp:boundedHc},~\ref{assp:Lsmoothc},~\ref{assp:consistentLICQ},~\ref{assp:sto-gk-essbounded} and~\ref{assp:gk-adaptive}, we have Lemma~\ref{lem:sto-eqcons-gk-norm},~\ref{lem:sto-eqcons-dk-diff}, \ref{lem:sto-lambda-bounded}, \ref{lem:sto-eqcons-beta-bounded}, \ref{lem:sto-eqcons-exp-decrease-1}, \ref{lem:sto-eqcons-exp-decrease}. The result is summarized in the following Lemma.

\begin{lemma}\label{lem:sto-eqcons-exp-decrease-adaptive}
        Under Assumption~\ref{assp:gk-adaptive},  let 
    \begin{equation} \label{eqn:sto-eqcons-exp-decrease-adaptive-1}
     \centering
     \begin{aligned}
        c_{\varphi} = [(1-\eta_{\beta}) \pi_l - \eta (\pi_u+1)] \frac{\alpha_k}{2}  - [(1-\eta_{\beta}) \pi_l+ c_k^{\mu}]\sqrt{\eta\alpha_k}  - \frac{1}{2\alpha_k}(\pi_u+1)  \eta.
      \end{aligned} 
    \end{equation} 
 Suppose by choosing appropriate parameters $\eta_{\beta}$, $\eta$, $\eta_{\alpha}$, $\{\alpha_k\}$, $\{\mu_k\}$, we have $c_{\varphi} > 0$. Then, there exists positive constants such as $c_{\varphi}$ so that 
\begin{equation} \label{eqn:sto-eqcons-exp-decrease-adaptive-2}
 \centering
  \begin{aligned}
	\Ebb_k\left[\varphi(x_k,\bar{\theta}_k)-\varphi(x_{k+1},\bar{\theta}_k)\right]> c_{\varphi} \nu_k \norm{d_k}^2.
  \end{aligned}
\end{equation}
\end{lemma}

\begin{proof}
         Using simple algebra, we know 
   \begin{equation} \label{eqn:sto-eqcons-exp-decrease-adaptive-pf-1}
     \centering
     \begin{aligned}
            \norm{\bar{d}_k}^2 = \norm{d_k + \bar{d}_k -d_k}^2 
                \geq& \norm{d_k}^2 + \norm{\bar{d}_k -d_k}^2 - 2 \norm{\bar{d}_k -d_k}\norm{d_k}.\\
     \end{aligned} 
    \end{equation}
   From Lemma~\ref{lem:sto-eqcons-dk-diff} and Assumption~\ref{assp:gk-adaptive}, taking $\Ebb_k$ of~\eqref{eqn:sto-eqcons-exp-decrease-adaptive-pf-1},
   \begin{equation} \label{eqn:sto-eqcons-exp-decrease-adaptive-pf-1.5}
     \centering
     \begin{aligned}
        \Ebb_k[\norm{\bar{d}_k}^2]  \geq& \norm{d_k}^2 - 2\frac{\sqrt{\eta\alpha_k}}{\alpha_k}\norm{d_k}^2
                 \geq(1-2\sqrt{\frac{\eta}{\alpha_k}})\norm{d_k}^2.  
     \end{aligned} 
    \end{equation}
          From~\eqref{eqn:sto-eqcons-KKT-9},~\eqref{assp:sto-grad-variance-adaptive} and~\eqref{eqn:sto-eqcons-exp-decrease-adaptive-pf-1.5}, we have 
   \begin{equation} \label{eqn:sto-eqcons-exp-decrease-adaptive-pf-2}
     \centering
     \begin{aligned}
	\Ebb_k [\varphi(x_k,\bar{\theta}_k)-\varphi(x_{k+1},\bar{\theta}_k)]> 
            & c_{1k} \nu_k \Ebb_k[\norm{\bar{d}_k}^2] - c_k^{\mu} \nu_k\sqrt{\eta\alpha_k}\norm{d_k}^2 - c_{2k} \nu_k \eta \alpha_k \norm{d_k}^2\\
           \geq & [c_{1k}(1-2\sqrt{\frac{\eta}{\alpha_k}}) -c_{k}^{\mu}\sqrt{\eta\alpha_k} -  c_{2k}\eta\alpha_k]\nu_k\norm{d_k}^2,
     \end{aligned} 
    \end{equation}
    where $c_{1k} = (1-\eta_{\beta})\frac{\alpha_k}{2} \pi_l$, $c_{2k} = \frac{\alpha_k^2+1}{2\alpha_k^2}(\pi_u+1)$.
    The right-hand side of~\eqref{eqn:sto-eqcons-exp-decrease-adaptive-pf-2} is 
   \begin{equation} \label{eqn:sto-eqcons-exp-decrease-adaptive-pf-3}
     \centering
     \begin{aligned}
        \left\{[(1-\eta_{\beta}) \pi_l - \eta (\pi_u+1)] \frac{\alpha_k}{2}  - [(1-\eta_{\beta}) \pi_l+ c_k^{\mu}]\sqrt{\eta\alpha_k}  - \frac{1}{2\alpha_k}(\pi_u+1)  \eta\right\} \nu_k \norm{d_k}^2.
      \end{aligned} 
    \end{equation} 
   Thus, the proof is complete.
\end{proof}
Apply Lemma~\ref{lem:super-martingale-thm} to~\ref{lem:sto-eqcons-exp-decrease-adaptive}, a convergence theorem follows. The proof is omitted due to the similarity to Theorem~\ref{thm:sto-eqcons-convg-2}.
\begin{theorem}\label{thm:sto-eqcons-convg-adaptive}
	Under the conditions of Lemma~\ref{lem:sto-eqcons-exp-decrease-adaptive}, if the sequence $\{\nu_k\}$ satisfies $\limsup_{k\to\infty} \nu_k > 0$,
       then with probability 1, every accumulation point of $\{x_k\}$ is a KKT point of~\eqref{eqn:opt-prob-sto}
\end{theorem}
As we mentioned above,~\eqref{assp:sto-grad-variance-adaptive} is not implementable and thus needs to be approximated in practice. 
For sample set  $S_k$ of  $\xi$ with realizations $\xi_i\in S_k$ at iteration $k$, 
the stochastic estimate of $g_k$ can be given as in~\eqref{eqn:gk-estimate} 
	 $\bar{g}_k = \frac{1}{N_k} \sum_{i=1}^{N_k} G(x_k,\xi_i)$,
where $G(x_k,\xi_i) \in \bar{\partial} R(x_k, \xi_i)$.
For nonsmooth objective $r(\cdot)$, a consistent form of subgradient $G(x_k,\xi_i)$ would be efficient in producing a non-biased estimate $\bar{g}_k$ of a $g_k$.
Assuming i.i.d. random variable sampling, a practical approximation of~\eqref{assp:sto-grad-variance-adaptive} is   
\begin{equation}\label{imp:sto-eqcons-gk-norm}
 \begin{aligned}
	 \frac{1}{|S_k|-1}\frac{\sum_{\xi_i \in S_k} \norm{ G(x_k,\xi_i) -\bar{g}_k}^2}{|S_k|} \leq \eta \alpha_k \norm{\bar{d}_k}^2.
 \end{aligned}
\end{equation}
The left-hand side of~\eqref{imp:sto-eqcons-gk-norm} is an unbiased estimate of $\Ebb_k[\norm{\bar{g}_k-g_k}^2]$.
The practical adaptive sampling strategy is given in Algorithm~\ref{alg:sto-adaptive}, which can be applied to Algorithm~\ref{alg:sto-convex} and~\ref{alg:sto-eqcons} for determining the size of $S_k$.
\begin{algorithm}
   \caption{Adaptive sample update algorithm}\label{alg:sto-adaptive}
	\begin{algorithmic}[1]
                \STATE{Generate sample sets $\{\xi_i\},\xi_i\in S_k$ i.i.d. from probability distribution of $\xi$.}
		\STATE{Compute the quantities in~\eqref{imp:sto-eqcons-gk-norm}.}
                \IF{\eqref{imp:sto-eqcons-gk-norm} stands}
                    \STATE{Set $N_{k+1}=N_k$.}
                \ELSE
                    \STATE{Set $N_{k+1} =  \frac{\sum_{\xi_i \in S_k} \norm{ G(x_k,\xi_i) -\bar{g}_k}^2}{\eta\alpha_k \norm{\bar{d}_k}^2 (N_k-1)}$.}
                \ENDIF
    \end{algorithmic}
\end{algorithm}

\section{Numerical Applications}\label{sec:exp}
We present two numerical examples to demonstrate the  
capabilities of the proposed algorithm. They are chosen within 
the general formulation of two-stage stochastic optimization problems.
As noted in previous sections, the parameter $\alpha_k$ requires knowledge of the objective functions.
In our examples, they are chosen initially to be $10$ times the known function value range and adjusted 
as the optimization progresses.

The first example is a joint production, pricing, and shipment problem 
that include both an online store and some offline physical stores. A similar problem is presented in~\cite{liu2020}.
The first-stage variable is the product price $p\in\Rbb$ and production quantity $x\in\Rbb$ for an online store.
The demand curve for the online store is assumed to be deterministic and the demand for physical stores stochastic.
The second-stage variables are the production for each factory/warehouse $i$ and the units shipped from factory/warehouse $i$ to physical store location $j$,
denoted as $y_i,i=1,\dots,M$ and $z_{ij},i=1,\dots,M,j=1,\dots,N$, respectively.
Given supply chain constraints, the last-minute production is deemed infeasible and a minimum production and 
storage for each physical and online store is required. 
The mathematical representation of this two-stage stochastic programming problem is 
\begin{example}\label{exmp:1}
\begin{equation}\label{prob:ex1-1st}
 \begin{aligned}
  &\underset{x,p }{\text{minimize}} 
	  & &  (c_1-p) x + \Ebb[ R(p,\xi)]\\
   &\text{subject to}
	  &&  10\geq p \geq 1, \ x\geq 1\\
	  &&& x \leq \alpha_0 p +\beta_0. 
 \end{aligned}
\end{equation}
	The function $R(p,\xi)$ is the value function to the second-stage problem  
\begin{equation}\label{prob:ex1-2nd}
 \begin{aligned}
	 R(p,\xi) \ = \ &\underset{y,z}{\text{minimize}}  && c_2^T y + \sum_{i=1}^M\sum_{j=1}^N (s_{ij}-p) z_{ij}\\
   &\text{subject to}
	  &&  \sum_{i=1}^M z_{ij} \leq \alpha_j(\xi)p + \beta_j(\xi) , \ j=1,\dots, N,\\
	  &&& \sum_{j=1}^N z_{ij} \leq y_i , \ i=1,\dots,M,\\
	  &&& z \geq 0, \ y\geq 1.
 \end{aligned}
\end{equation}
\end{example}
The number of factories is $M=5$ and the number of physical stores is $N=5$. The first-stage cost per unit is $c_1=4.2$, reflecting the cost of both production and shipping.
For the second-stage, the unit production cost is $c_2=[2.2,3.2,3.3,4.2,2.4]^T$ and the 
unit shipment cost from factory $i$ to store $j$, denoted as $s_{ij}$, is $s_{ij}=2$ for all $i=1,\dots,5$ and $j=1,\dots,5$.
The first-stage demand slope is defined by $\alpha_0 = -1$ and $\beta_0=12$.
The random demand slope $\alpha_j$ at store $j$ is generated from truncated normal distribution on $[-1.5,-0.5]$, $[-2,-1]$, $[-2.5,-1.5]$, $[-3,-2]$ and $[-2.5,-1.5]$ for $j=1,\dots,5$. The random intercepts follow truncated normal distribution on $[16,17]$, $[21,22]$, $[26,27]$, $[31,32]$ and $[26,27]$. The price variable is bounded, reflected as $p\in[1,10]$. 
Additionally, we set a lower bound on $x$ to keep the factory/warehouse active. 

The sample set at iteration $k$ is $S_k=\{\xi_1,\dots,\xi_{N_k}\}$.
The second-stage problem is coupled both in constraint and objective with $p$. 
The function $R(\cdot,\xi)$ is upper-$\Ctwo$ in $p$.
To illustrate this numerically, we set $N_k=1000$ and compute the estimate of $\Ebb[R(p,\xi)]$ with  
$ \frac{1}{N_k} \sum_{i=1}^{N_k} R(p,\xi_i)$. 
Moreover, given the affine nature of the second-stage problem, it is possible to compute the analytic (true) expression of $\Ebb[R(p,\xi)]$ and $g_k\in \bar{\partial}r(x_k)$ for $r(p)$. Both the true value function and subgradient are plotted in Figure~\ref{fig:ex_ssp}. The true subgradient $g_k$ is used in establishing the error measure as well.

As shown in Figure~\ref{fig:ex_ssp}, the function $r(p)$ is nonsmooth nonconvex and upper-$\Ctwo$. Further, an accurate estimate of it can be established with sufficient number of sample points. 
The constraints of the first-stage problem are affine, and therefore Algorithm~\ref{alg:sto-convex} and~\ref{alg:sto-adaptive} can be applied. The parameters of the algorithms are $\alpha_0=15$, $\eta=1$, $\eta_{\alpha}=1.5$, and  $x_0=[1.5,1.5]$. 
\begin{figure}
  \centering
  \includegraphics[width=0.9\textwidth]{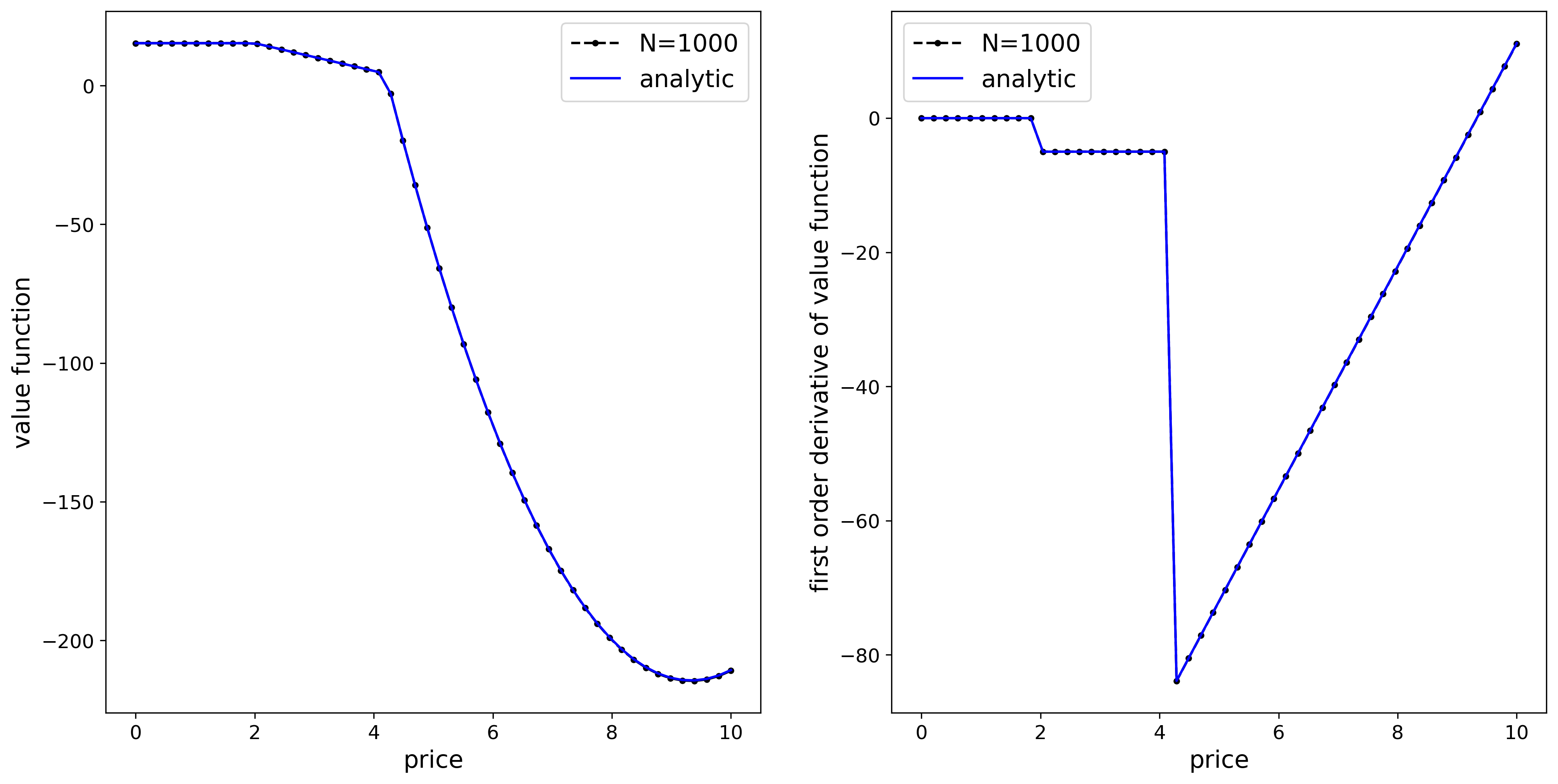}
	\caption{Value function and subgradient with analytic expression and $N_k=1000$ sample average}
\label{fig:ex_ssp}
\end{figure}

We consider five different sampling strategies for Algorithm~\ref{alg:sto-convex}. Three of them have fixed $N_k$ at values $10,100,1000$. The fourth one uses a predetermined increasing sample size $N_k=k^{1.25}$, which corresponds to a decreasing (and summable) $M_k$.  The last one uses~\eqref{imp:sto-eqcons-gk-norm} to adaptively update $N_k$. Both varying sample sizes are capped at $1000$. 
The error at each step is computed based on the stationary measure $\norm{g_k- \nabla c(x_k) \lambda^{k}}$ with the true subgradient $g_k$. The constraints $c(x_k)$ come from the inequality constraints in~\eqref{prob:ex1-1st}  and the optimal Lagrange multipliers $\lambda^{k+1}\in\Rbb^m$ are obtained by solving a least-square optimization problem  
\begin{equation}\label{prob:ex1-lam}
 \begin{aligned}
  &\underset{\lambda}{\text{minimize}} 
	  & &  \norm{g_k - \nabla c(x_k)\lambda }^2\\
   &\text{subject to}
	  &&  c_j(x_k) \lambda_j =0, j=1,2,\cdots,m\\
	  &&& \lambda \geq 0. 
 \end{aligned}
\end{equation}
Each sampling strategy is given a budget of $50000$ second-stage problem solutions and we call every $500$ second-stage solves an epoch. Each run is repeated five times and average values of quantities of interest are used as results. The convergence and cost results are plotted in Figure~\ref{fig:ex_ssp_error}, where the fourth strategy is marked as $m_k$ and the fifth one $adaptive$. 

The result indicates that SQP with fixed sample sizes can solve nonsmooth problems with upper-$\Ctwo$ objectives to certain accuracies as expected. The larger the sample size, the smaller the average optimality error. This is also observed for the fourth sampling strategy as its sample size gradually increases to around $650$ at the end of the computing budget, and its average error continues to decrease. Compared to the adaptive sampling result, its sample size increases more slowly, controlled by our choice of predetermined sequence $N_k=k^{1.25}$.  
Comparing the adaptive sampling strategy to SAA with $N_k=1000$, the same level of optimality measure can be reached by the former with much fewer overall second-stage problem solutions and subgradient evaluations. 
This is reflected by the fact that $N_k$ for adaptive sampling increased from $2$ to the maximum value $1000$ at around $22$ iterations (see Figure~\ref{fig:ex_sample_size}) and the error peak occurs much sooner in the epoch plot of Figure~\ref{fig:ex_ssp_error} 
compared to the $N_k=1000$ SAA. The same conclusion can be drawn comparing the adaptive (fifth) to predetermined (fourth) sampling strategies. 
We note that for the fourth strategy, it is certainly possible to design a more problem-dependent sequence $N_k$ that can outperform the adaptive sampling strategy in terms of efficiency and accuracy. 
\begin{figure}
  \centering
  \includegraphics[width=0.95\textwidth]{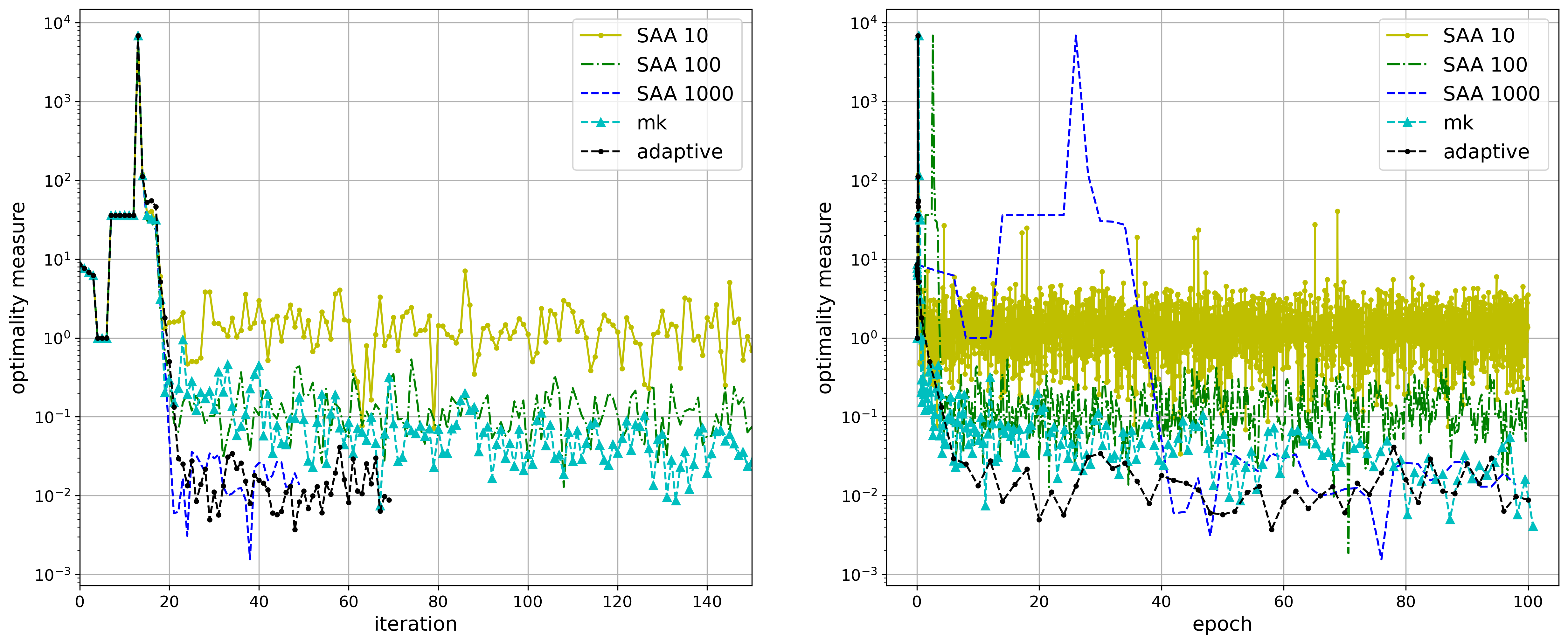}
	\caption{Convergence and cost of Example 1 (Algorithm~\ref{alg:sto-convex}) with different sampling strategies.} 
\label{fig:ex_ssp_error}
\end{figure}

In the second example, we apply the algorithm with adaptive sampling to SCACOPF problems, given the deterministic SAA SQP success in~\cite{wang2022}.
\begin{example}\label{exmp:2}
	(smoothed SCACOPF) Example 2 is a SCACOPF problem with coupling affine active power constraints for 
	contingency (second-stage) problems. The network data used in this example is from the ARPA-E Grid Optimization competition~\cite{petra_21_gollnlp}. 
	The full mathematical formulation is complex but the base (first-stage) problem fits in the form of~\eqref{eqn:opt-prob-sto}, 
	where $r(\cdot)$ is the expectation of the recourse function of the contingency problems. 
	Details of the 
	problem setup can be seen in~\cite{petra_21_gollnlp}. 
        Using a quadratic penalty of the coupling constraints in the contingency problems, $R(\cdot,\xi)$ becomes upper-$\Ctwo$ in $x$.
	The problem is hence referred to as the smoothed SCACOPF.
\end{example}

For the purpose of demonstration, we consider a set of $350$ contingencies that follow discrete uniform distribution. 
The random variable $\xi$ models an integer that represents one of the contingencies.  
The objective $r(\cdot)$ is thus the average of $R(x,\xi_i),i=1,\cdots,350$ and is upper-$\Ctwo$.
Hence, the true objective and subgradient can be obtained by going through exhaustively the complete set of contingencies. 
We point out that for contingency problems that employs more complex probability distribution, the true objective would be unavailable.
Based on the potential values of the objective, the parameters of Algorithm~\ref{alg:sto-eqcons} and~\ref{alg:sto-adaptive} are set to $\alpha_0=5\times 10^6$, $\eta_{\alpha}=1.5$, $\gamma=10$, $\eta_{\beta}=0.2$, $\nu_k=1$, $\mu_k=0$ and $\eta=10^5$. The large $\eta$ value is chosen based on $\alpha_k$.

Four sampling strategies are compared with three fixed sample sizes $15,50,150$ and the adaptive sampling one. An upper limit of $150$ is imposed on $N_k$ in the adaptive sampling algorithm. All four tests run for at least 200 iterations.  
The constraints are satisfied by all four runs to acceptable levels.
At each iteration of the algorithms, the true objectives are evaluated and plotted in Figure~\ref{fig:ex_scacopf}. 
It is clear that increasing the sample size provides a more stable reduction in the objective value as expected. 
With a sample size of $150$, which is smaller than the number of discrete random values, the oscillation from the stochastic algorithm reaches a tolerable level.  Smaller sample sizes still manage to reduce objective but cannot produce stabilized results. 
Figure~\ref{fig:ex_scacopf-epoch} plots the first $4500$ contingency problem solutions (epoch) and it is clear that 
the adaptive sampling strategy successfully generates decrease rather quickly while following a stable path later on as designed.
\begin{figure}
  \centering
  \includegraphics[width=0.9\textwidth]{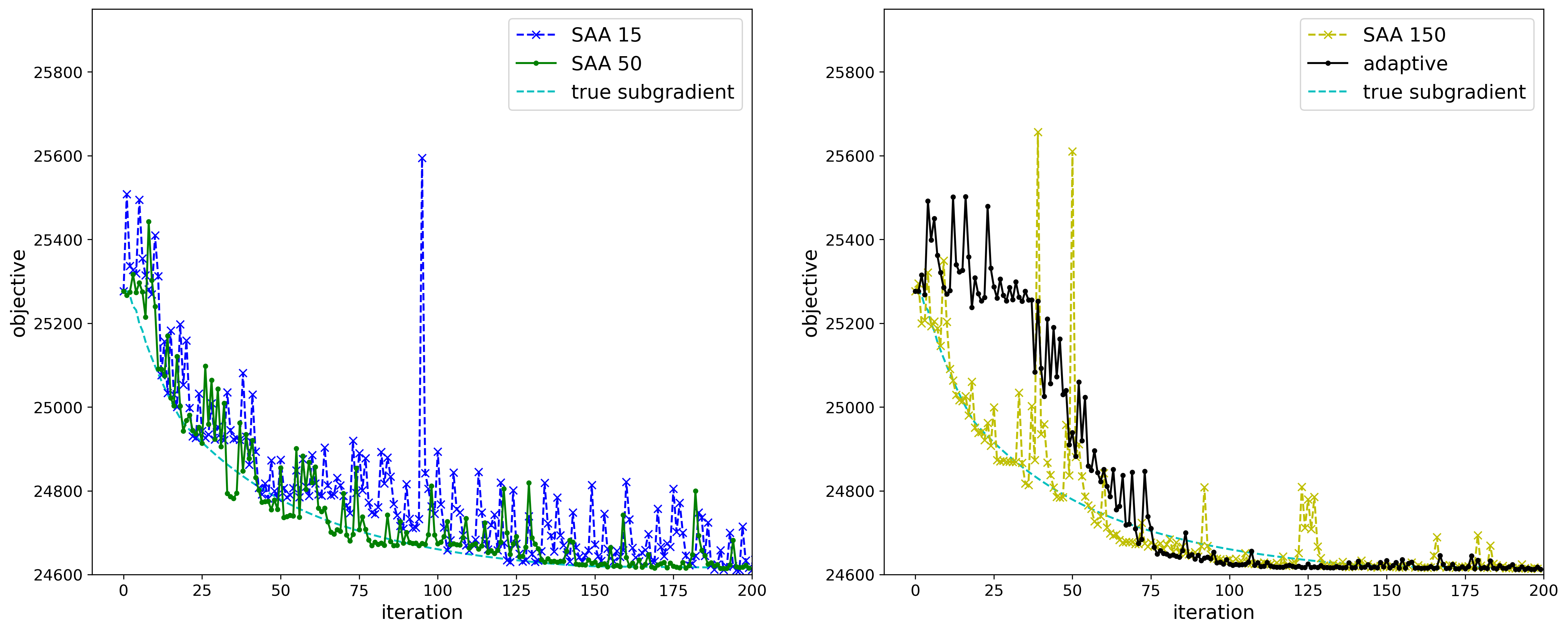}
	\caption{Objective \textit{v.s.} iteration for Example 2 with different sampling strategies.} 
\label{fig:ex_scacopf}
\end{figure}
\begin{figure}
  \centering
  \includegraphics[width=0.85\textwidth]{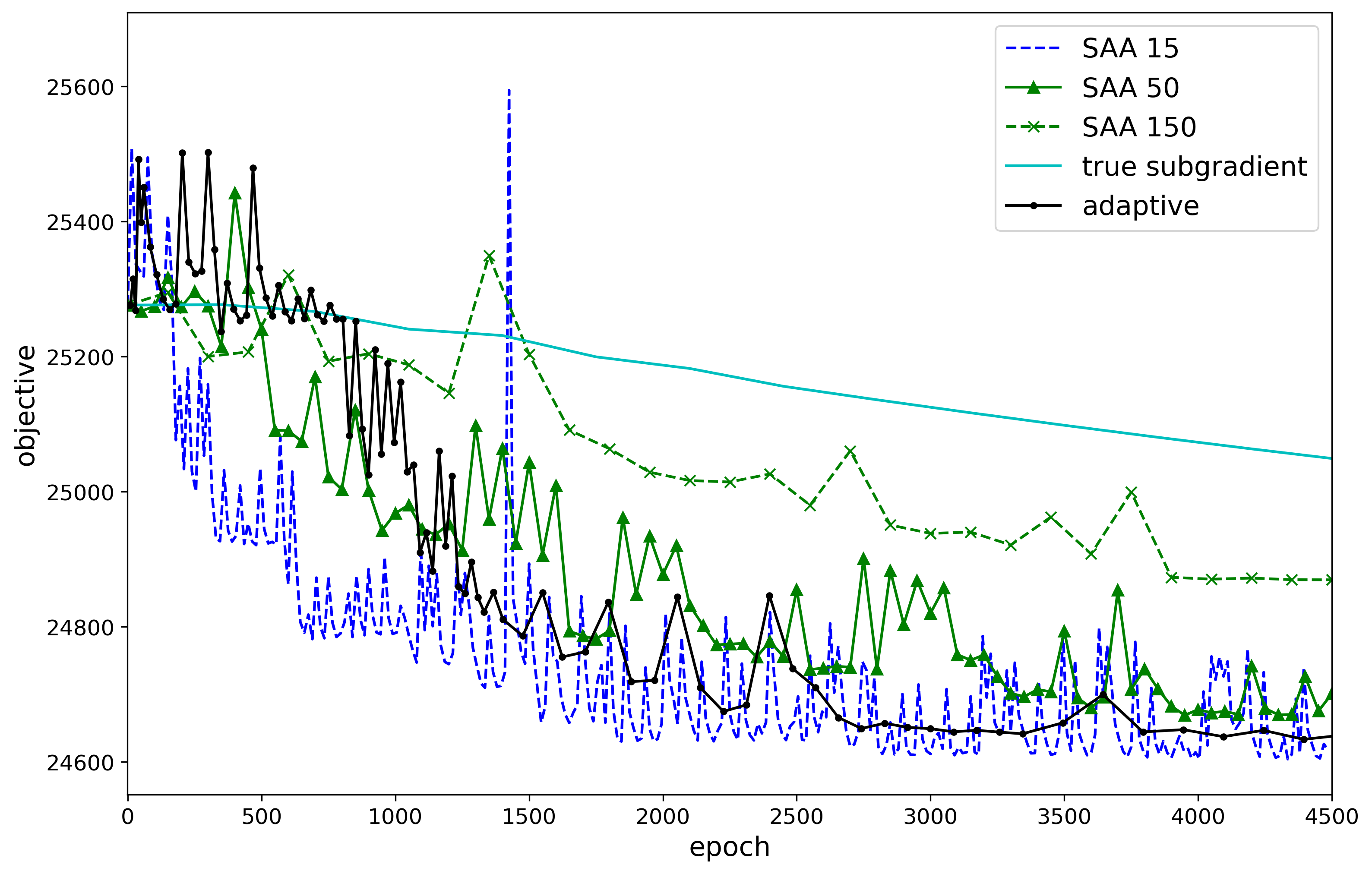}
	\caption{Objective \textit{v.s.} epoch for Example 2 with different sampling strategies.} 
\label{fig:ex_scacopf-epoch}
\end{figure}

\begin{figure}
  \centering
  \includegraphics[width=0.9\textwidth]{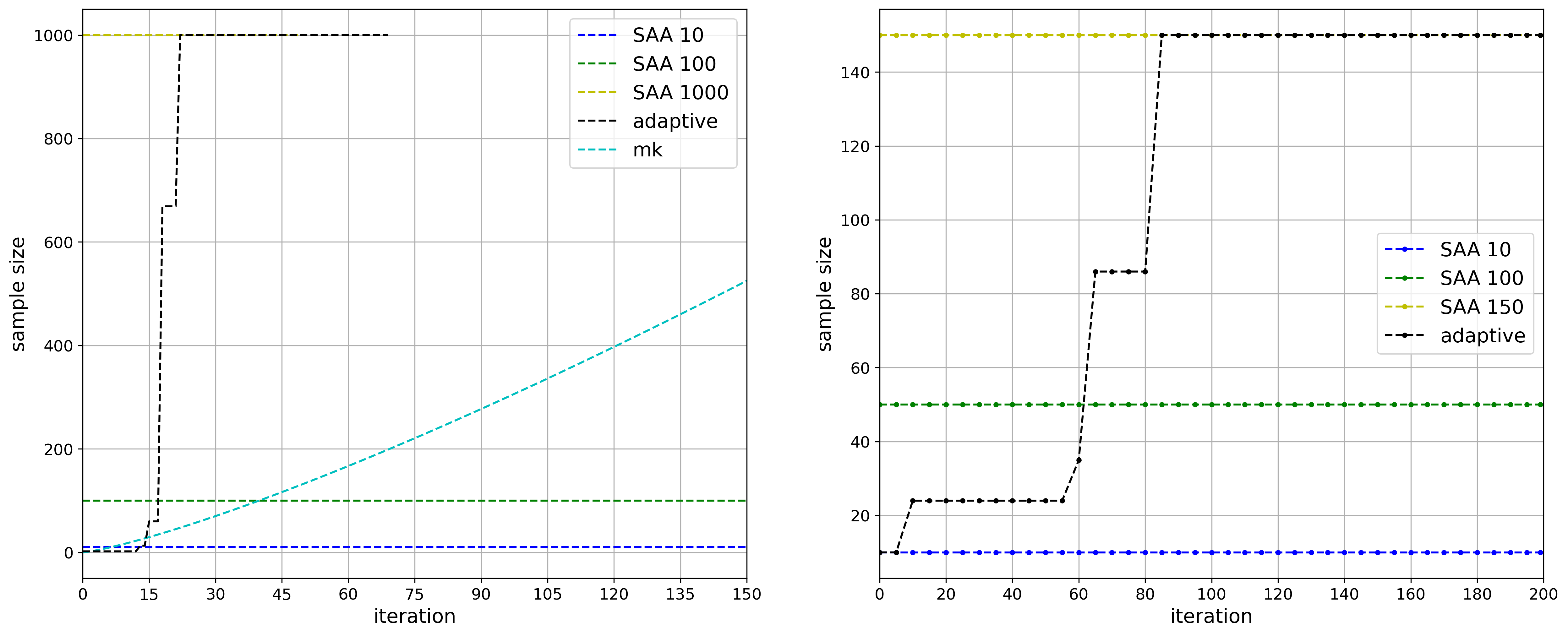}
	\caption{Sample size plots for both examples.} 
\label{fig:ex_sample_size}
\end{figure}
The evolution of sample size $N_k$ for both examples are shown in Figure~\ref{fig:ex_sample_size}.
The result of example 2 encourages the application of adaptive sampling to SCACOPF problems that 
might be too large to be solved conventionally or that have a more complicated probability distribution for $\xi$.

\section{\normalsize Conclusions}\label{sec:con}
In this paper, we have proposed and analyzed SQP algorithms for constrained stochastic nonsmooth nonconvex optimization problems
with upper-$\Ctwo$ objectives. A range of important problems fit the mathematical setup, particularly two-stage stochastic optimization problems. 
The proposed algorithms help to fill the gap of nonsmooth stochastic optimization algorithms with adaptive sampling. Furthermore, problems with and without smooth equality constraints are discussed separately with two distinct algorithms.
Both algorithms formulate a convex quadratic programming subproblem at each iteration based on 
stochastic approximation of the objective. The equality constrained problem requires additional, carefully designed line search to ensure overall progress.
Subsequential convergence analysis of the proposed pair of algorithms with respect to expectation is provided using widely adopted assumptions.
The adaptive sampling algorithm can be used in practice to improve efficiency of the algorithm.


\appendix

\section*{Acknowledgments}
This work was performed under the auspices of the U.S. Department of Energy by Lawrence Livermore National Laboratory under contract DE--AC52--07NA27344.  
\bibliographystyle{siamplain}
\bibliography{bibliography}
\end{document}